\documentclass{article}
\usepackage{CJK,CJKnumb,CJKulem,times,dsfont,ifthen,mathrsfs,latexsym,amsfonts,color}
\usepackage{amsmath,amsthm,makeidx,fontenc,amssymb,bm,graphicx,psfrag,listings,curves,extarrows,enumitem}
\usepackage{hyperref}

\usepackage{multirow} 
\usepackage{wrapfig}
\usepackage{float}
\usepackage{geometry}
\usepackage{indentfirst}
\usepackage{amssymb}
\makeatletter

\newcommand{\Rmnum}[1]{\expandafter\@slowromancap\romannumeral #1@}
\makeatother

\usepackage[showonlyrefs]{mathtools}  
\mathtoolsset{showonlyrefs=true}

\newtheorem{theorem}{Theorem}[section]
\newtheorem{lemma}{Lemma}[section]

\newtheorem{proposition}{Proposition}[section]
\newtheorem{remark}{Remark}[section]

\newcommand{\C}{{\mathbb C}}

\newcommand{\R}{{\mathbb R}}
\newcommand{\bean}{\begin{eqnarray*}}
	\newcommand{\eean}{\end{eqnarray*}}

\newcommand{\sbr}[1]{\left(#1\right)}
\newcommand{\mbr}[1]{\left[#1\right]}
\newcommand{\lbr}[1]{\left\{#1\right\}}

\setitemize{itemindent=38pt,leftmargin=0pt,itemsep=-0.4ex,listparindent=26pt,partopsep=0pt,parsep=0.5ex,topsep=-0.25ex}

\numberwithin{equation}{section}

\begin{document}
	\theoremstyle{plain}

	\title{\bf  Positive  solution  for an elliptic system with critical exponent and logarithmic terms
		\thanks{Supported NSFC(No.12171265). E-mail addresses: hhajaie@calstatela.edu (H. Hajaiej), liuth19@mails.tsinghua.edu.cn (T. H. Liu),  songlinjie18@mails.ucas.ac.cn (L. J. Song),  zou-wm@mail.tsinghua.edu.cn (W. M. Zou).} }
	
\author{Hichem Hajaiej$^{\mathrm{a,}}$, Tianhao Liu$^{\mathrm{b,}}$, Linjie Song$^{\mathrm{c,d,}}$, Wenming Zou$^{\mathrm{e,}}$ \thanks{%
The author is supported by CEMS. Email: songlinjie18@mails.ucas.edu.cn.}\  \\
\\
{\small $^{\mathrm{a}}$ Department of Mathematics, California State University at Los Angeles, Los Angeles, CA 90032, USA}\\
{\small $^{\mathrm{b,e}}$ Department of Mathematical Sciences, Tsinghua University, Beijing 100084}\\
{\small $^{\mathrm{c}}$Institute of Mathematics, AMSS, Academia Sinica,
Beijing 100190, China}\\
{\small $^{\mathrm{d}}$University of Chinese Academy of Science, Beijing 100049, China}
}
	\date{}
	
	\maketitle
	
			\begin{center}{\bf Abstract }	\end{center}

			 In this paper, we study the existence and nonexistence  of  positive solutions   for the following coupled elliptic system  with  critical exponent and logarithmic terms:
			 \begin{equation}
			 	\begin{cases}
			 		-\Delta u=\lambda_{1}u+
			 		\mu_1|u|^{2}u+\beta |v|^{2}u+\theta_1 u\log u^2, & \quad x\in \Omega,\\
			 		-\Delta v=\lambda_{2}v+
			 		\mu_2|v|^{2}v+\beta |u|^{2}v+\theta_2 v\log v^2,  &\quad x\in \Omega,\\
			 		u=v=0, &\quad  x \in \partial \Omega,
			 	\end{cases}
			 \end{equation}
			 where $\Omega \subset \R^4$ is a bounded smooth domain,  the parameters $\lambda_{1},\lambda_{2},\theta_{1},\theta_{2}\in \R  $, $\mu_{1},\mu_2>0$ and $\beta \neq 0$ is a coupling constant. Note that the logarithmic term $s\log s^2$ has special properties, which makes the problem more complicated. We  show that this system has a positive least energy solution for  $|\beta|$ small and  positive large $\beta$ if $\lambda_{1},\lambda_{2}\in \R$ and  $\theta_{1},\theta_{2}>0$. While the situations for the case  $\theta_{1},\theta_{2}<0$  are quite thorny, in this challenging setting we establish the existence result of positive local minima solutions and nonnegative solutions under various conditions on the parameters. Besides, under some further assumptions, we obtain the nonexistence of positive solutions for both the case where $\theta_{1}$,$\theta_{2}$ are
			 negative  and  the case  where they   have opposite signs.  Comparing our results with those of Chen and Zou  (Arch. Ration. Mech. Anal. 205:515--551, 2012), the logarithmic term  $s\log s^2$ introduces some new interesting phenomenon. Moreover, its presence brings major challenges and make it difficult to use the comparison theorem used in the work of Chen and Zou without new ideas and innovative techniques.   
			  To the best of our knowledge, our paper is  the first  to give a rather complete picture for the existence and nonexistence results to the coupled elliptic system  with   critical exponent and  logarithmic terms.
			  Also, we consider the related single equation
			 \begin{equation}
			 	-\Delta u=\lambda u + \mu|u|^{2}u+\theta u\log u^2, ~u\in H_0^1(\Omega)
			 \end{equation}
			 with $\mu>0$, $\theta<0$, $\lambda\in\R$ or $\lambda\in [0,\lambda_{1}(\Omega))$ and prove the existence of  the  positive  solution under some further suitable assumptions, which is the  type of  a local minima  or a  least energy solution.

			 \vskip0.23in
			
			
			 \vskip0.1in


	\section{Introduction} \label{Sect1}
 
	Consider the following coupled elliptic system  with   logarithmic terms
	\begin{equation} \label{System1}
		\begin{cases}
			-\Delta u=\lambda_{1}u+
			\mu_1|u|^{2}u+\beta |v|^{2}u+\theta_1 u\log u^2, & \quad x\in \Omega,\\
			-\Delta v=\lambda_{2}v+
			\mu_2|v|^{2}v+\beta |u|^{2}v+\theta_2 v\log v^2,  &\quad x\in \Omega,\\
			u=v=0, &\quad  x \in \partial \Omega,
		\end{cases}
	\end{equation}
	where $\Omega \subset \R^4$ is a bounded smooth domain,  $\lambda_{i},\theta_i\in \R  $ $(i=1,2)$, $\mu_{1},\mu_2>0$ and $\beta \neq 0$ is a coupling constant.

	 System \eqref{System1} is closely related to the following time-dependent nonlinear logarithmic type Schr\"odinger system
	\begin{equation} \label{System5}
		\begin{cases}
			\imath \partial_t \Psi_1	=\Delta \Psi_1 +
			\mu_1|\Psi_1|^{2}\Psi_1+\beta |\Psi_2|^{2}\Psi_1+\theta_1 \Psi_1\log \Psi_1^2 , \quad x\in \Omega,~ t>0,\\
			\imath \partial_t \Psi_2	=\Delta \Psi_2 +\
			\mu_2|\Psi_2|^{2}\Psi_2+\beta |\Psi_1|^{2}\Psi_2 +\theta_2 \Psi_2\log \Psi_2^2, \quad x\in \Omega, ~t>0,\\
			\Psi_i=\Psi_i(x,t)\in \C,\quad i=1,                                                                                                                    2\\
			\Psi_i(x,t)=0,\quad  x\in \partial \Omega,\quad t>0, \quad i=1,2,
		\end{cases}
	\end{equation}
	where $\imath$ is the imaginary unit. System \eqref{System5} appears in many physical fields, such as   quantum mechanics, quantum optics,
	nuclear physics, transport and diffusion phenomena,  open quantum systems, effective quantum gravity, theory
	of superfluidity and Bose-Einstein condensation. We  refer the readers to the papers \cite{Alfaro=DPDE=2017,Bialynicki=1975,Bialynicki=1976,Carles=2018,Carles=2014,Colin=2004,Poppenberg=2002,Wang=ARMA=2019}  for a survey on the related physical backgrounds.

	When $\theta_1=\theta_2=0$, the system  \eqref{System1}   reduces to the following coupled nonlinear Schr\"odinger system
	\begin{equation} \label{System3}
		\begin{cases}
			-\Delta u=\lambda_{1}u+
			\mu_1|u|^{2}u+\beta |v|^{2}u , &\quad x\in \Omega,\\
			-\Delta v=\lambda_{2}v+
			\mu_2|v|^{2}v+\beta |u|^{2}v , &\quad x\in \Omega,\\
			u=v=0, &\quad  x \in \partial \Omega ,
		\end{cases}
	\end{equation}
	which  appears when looking for standing wave solutions $\Psi_1(x,t)=e^{\imath \lambda_{1}t}u (x)$ and $\Psi_2(x,t)=e^{\imath \lambda_{2}t}v (x)$  of time-dependent coupled nonlinear Schr\"odinger system
	
	\begin{equation} \label{System4}
		\begin{cases}
			\imath \partial_t \Psi_1	=\Delta \Psi_1 +
			\mu_1|\Psi_1|^{2}\Psi_1+\beta |\Psi_2|^{2}\Psi_1 , \quad x\in \Omega,~ t>0,\\
			\imath \partial_t \Psi_2	=\Delta \Psi_2 +\
			\mu_2|\Psi_2|^{2}\Psi_2+\beta |\Psi_1|^{2}\Psi_2 , \quad x\in \Omega, ~t>0,\\
			\Psi_i=\Psi_i(x,t)\in \C,\quad i=1,                                                                                                                    2\\
			\Psi_i(x,t)=0,\quad  x\in \partial \Omega,\quad t>0, \quad i=1,2.
		\end{cases}
	\end{equation}
	This system \eqref{System4} originates from many physical models, especially in nonlinear optics and Bose-Einstein condensation. We refer for these to  \cite{Akhmediev1999,Esry1997,Frantzeskakis2010,Kivshar1998}, which also contain information about the physical relevance
	of non-cubic nonlinearities.

	Note  that the  nonlinearity and coupling terms  of system \eqref{System3}  are both  critical in dimension 4 (that is, the Sobolev critical exponent $2^*=\frac{2N}{N-2}=4$ when $N=4$), and the critical system \eqref{System3} was considered by Chen and Zou   \cite{Zou 2012} for the first time.  In \cite{Zou 2012}, they  proved that  there exist $0<\beta_*<\beta^*$ such that system \eqref{System3} has a positive least energy  solution if $\beta\in (-\infty, 0)\cup (0,\beta_* ) \cup (\beta^*, +\infty)$ when $0 <\lambda_{1},\lambda_{2}< \lambda_{1}(\Omega)$ and $\mu_1,\mu_2>0$,  where $\lambda_{1}(\Omega)$ is the first eigenvalue of $-\Delta$ with Dirichlet boundary conditions. We would also like to highlight the paper \cite{Zou 2015,Liu-You-Zou=arxiv=2022,YePeng2014} where it is shown, for more general powers, that the dimension has a significant impact on the existence of positive least energy  solutions.
	
	Inspired by their work, and noticing  that   the system \eqref{System1} can be viewed as system \eqref{System3} with a logarithmic perturbation,  we are interested in the existence and nonexistence of positive solution for the  system \eqref{System1} with   logarithmic terms. To the best of our knowledge, there are no  results in the literature dealing with \eqref{System1}.

	The logarithmic term has some special properties, which make things more challenging. It is easy to see that
	\begin{equation}
		\lim\limits_{u\to 0^+} \frac{u\log u^2}{u}=-\infty,
	\end{equation}
	that is, $ u=o( u\log u^2)$ for $u$ very close to 0. Comparing with the critical term $|u|^2u$, the logarithmic term $u\log u^2$ is a lower-order term at infinity. Also, the sign of the logarithmic term changes depending on  $u$.    Moreover, the presence of logarithmic term makes the structure of the corresponding functional complicated.
	
	%
	
	Recently, Deng et al.\cite{Deng-He-Pan=Arxiv=2022} investigated the existence of  positive least energy solutions for the following related single equation
	\begin{equation}\label{Equation1}
		-\Delta u=\lambda_{i} u + \mu_{i}|u|^{2}u+\theta_{i}u\log u^2, ~u\in H_0^1(\Omega), ~ \Omega \subset \R^4,~ i=1,2,
	\end{equation}
	In \cite{Deng-He-Pan=Arxiv=2022}, the authors  proved that the equation \eqref{Equation1} has  a positive least energy solution if $\lambda_{i}\in \R$, $\mu_{i},\theta_{i}>0$. Also, they obtained the existence and nonexistence results under  other different conditions, we refer the readers to \cite{Deng-He-Pan=Arxiv=2022} for details.  When $\theta_{i}=0$, the equation \eqref{Equation1} reduces to the classical 	Br\'ezis-Nirenberg  problem (see \cite{Brezis-Nirenberg1983}). It is well known that the classical 	Br\'ezis-Nirenberg  equation has  a positive least energy solution if $0<\lambda_{i}<\lambda_{1}(\Omega)$, $\mu_{i}>0$.  So, the logarithmic term $\theta_{i} u\log u^2$ has much more influence than the term $\lambda_{i} u$ on the existence solutions. Therefore, we believe that  some different phenomena  will occur	when  studying  the  existence of  positive least energy solution for the system \eqref{System1}.

	\vskip 0.1in
	In this paper, we will study  the existence and nonexistence of positive   solution for the  system \eqref{System1} with   logarithmic terms. The fact that $\theta_{i}\neq 0$ makes the problem quite different and also more
	complicated comparing to the case $\theta_{i}=0$ ($i=1,2$), and requires us to develop some different ideas and techniques.
	
	Let us first introduce some definitions.
	We call a solution $(u,v)$ fully nontrivial if both $u\not \equiv 0$ and $v\not \equiv 0$; we call a solution $(u,v)$ semi-trivial  if  $u\equiv0 $, $v\not \equiv 0$ or $u\not \equiv0 $, $v \equiv 0$; we call a solution $(u,v)$ positive (resp. nonnegative) if both $u>0$ and $v>0$ (resp. both $u\geq 0$ and $v\geq0$ ).

	Define $\mathcal{H}:=H_0^1(\Omega)\times H_0^1(\Omega)$. To find a  positive solution to the system \eqref{System1}, we borrow the ideas from \cite{Deng-He-Pan=Arxiv=2022} to define a modified functional $\mathcal{L}:\mathcal{H} \to \R$
	\begin{equation} \label{Fun L}
		\begin{aligned}
			\mathcal{L}(u,v)&=\frac{1}{2}\int_{\Omega} |\nabla u|^2-\frac{\lambda_1}{2}\int_{\Omega}|u^+|^2-\frac{\mu_1}{4} \int_{\Omega}|u^+|^4-\frac{\theta_1}{2}\int_{\Omega}(u^+)^2(\log (u^+)^2-1)\\&+\frac{1}{2}\int_{\Omega} |\nabla v|^2-\frac{\lambda_2}{2}\int_{\Omega}|v^+|^2-\frac{\mu_2}{4} \int_{\Omega}|v^+|^4-\frac{\theta_2}{2}\int_{\Omega}(v^+)^2(\log (v^+)^2-1)-\frac{\beta}{2}\int_{\Omega}|u^+|^2|v^+|^2 ,
		\end{aligned}
	\end{equation}
	where $u^+:=\max\lbr{u,0}$, $u^-:=-\max \lbr{-u,0}$. It is easy to see that the functional $\mathcal{L}$ is well-defined in $\mathcal{H} $. Moreover,  any nonnegative critical point of $\mathcal{L}$ corresponds to a solution of the system \eqref{System1}.
	
	We say a solution $(u,v)$ of \eqref{System1} is a least energy solution if $(u,v)$ is fully nontrivial and $\mathcal{L}(u,v)\leq\mathcal{L}(\varphi,\psi) $ for any other fully nontrivial solution $(\varphi,\psi)$ of system \eqref{System1}. As in \cite{Lin-Wei=CMP=2005}, we consider
	\begin{equation}\label{Nehari manifold}
		\mathcal{N}=\lbr{(u,v)\in \mathcal{H}:u\not \equiv 0, v\not \equiv 0, \mathcal{L}^\prime(u,v)(u,0)=0 \text{ and }\mathcal{L}^\prime(u,v)(0,v)=0}.
	\end{equation}
	It is easy to check that  $\mathcal{N}\neq \emptyset$. Then  we set
	\begin{equation}
		\mathcal{C}_{\mathcal{N}}:=\inf_{(u,v)\in \mathcal{N}} \mathcal{L}(u,v).
	\end{equation}
	Define
	\begin{equation}
		\varSigma_1:=\lbr{(\lambda,\mu,\theta):\lambda\in\R,\mu>0,\theta>0},
	\end{equation}
	and
	\begin{equation} \label{defi of Lambda}
		\Lambda:=\min_{i=1,2}\lbr{\cfrac{\mu_i}{\sbr{\mu_i+\theta_i e^{\lambda_i/\theta_i-1}}^2}}.
	\end{equation}
	Our first existence result of this paper is the following.
	\begin{theorem}\label{Theorem-1}
		Assume that  $\sbr{\lambda_i, \mu_i, \theta_i}\in \varSigma_1~$ for  $i=1,2$.
		\begin{enumerate}
			\item[(1)] There exists  $\beta_0\in \left( 0,\min\lbr{\mu_1,\mu_2}\right] $ such that system \eqref{System1} has a positive least energy solution  $(u,v)\in \mathcal{N}$ with  $\mathcal{L}(u,v)=\mathcal{C}_{\mathcal{N}}$ for any $\beta \in  \sbr{-\beta_0,0} $.
			\item[(2)] There exists $\beta_1\in \left( 0,\min\lbr{\mu_1,\mu_2}\right] $ such that system \eqref{System1} has a positive least energy solution  $(u,v)\in \mathcal{N}$ with  $\mathcal{L}(u,v)=\mathcal{C}_{\mathcal{N}}$ for any $\beta \in \sbr{0,\beta_1}$.
			\item[(3)] Let $\beta_2$ be the larger root of the equation
		\begin{equation} \label{f54}
			\beta^2-\frac{2}{\Lambda}\beta +\frac{\mu_1+\mu_2}{\Lambda}-\mu_1\mu_2=0,
		\end{equation}
		where $\Lambda$ is defined in \eqref{defi of Lambda}. Then we have $\beta_2>\max\lbr{\mu_1,\mu_2}$, and system \eqref{System1} has a positive least energy solution  $(u,v)\in \mathcal{N}$ with  $\mathcal{L}(u,v)=\mathcal{C}_{\mathcal{N}}$  for any $\beta>\beta_2$.
		\end{enumerate}
	\end{theorem}

   	\begin{remark}
   	{\rm To the best of our knowledge, Theorem \ref{Theorem-1} seems to be the first result devoted to the existence of solutions for the coupled elliptic system with logarithmic terms.  Comparing the existence  results of \cite[Theorem 1.3]{Zou 2012} and Theorem\ref{Theorem-1} above,   we list the conditions in the table below for clarity.
   		
   		\begin{table}[H]
   			\vspace{-1.0em}
   			\begin{center}
   				\caption{ The existence results under different conditions for $\mu_i>0$, $i=1,2$.}
   				\vskip0.1in
   				\begin{tabular}{|c|c|c|c|c|c|}
   					\hline
   					\multicolumn{2}{|c|}{$\beta<0$} &	\multicolumn{2}{|c|}{$\beta>0$ small} & \multicolumn{2}{|c|}{$\beta>0$ large}\\
   					\hline
   					$\theta_i=0$ & $\theta_i>0$&$\theta_i=0$ & $\theta_i>0$&$\theta_i=0$ & $\theta_i>0$\\
   					\hline
   					$ \lambda_{i}\in \sbr{0,\lambda_{1}(\Omega)}$  & $\lambda_i\in \R$ & $ \lambda_{i}\in \sbr{0,\lambda_{1}(\Omega)}$ &$\lambda_i \in \R$ &	$ \lambda_{i}\in \sbr{0,\lambda_{1}(\Omega)}$ & $ \lambda_{i}\in \R$ \\
   					\hline
   					$\beta<0$ &$\beta\in  \sbr{-\beta_0,0} $ &$\beta\in \sbr{0,\beta_*}$ & $\beta\in\sbr{0,\beta_1} $ &$\beta>\beta^*$ &$\beta>\beta_2$\\
   					\hline
   				\end{tabular}
   			\end{center}
   			\vspace{-2.0em}
   		\end{table}
   		Here,   $\beta_0,\beta_1,\beta_2, \beta_*, \beta^*$  are some prescribed constants.	
   		Therefore, we observe that if $\theta_i>0$, the dominant contribution to the existence of the positive solution comes from the logarithmic terms $\theta_{i} u \log u^2$ for  $  i=1,2$, rather than  $\lambda_{i} u$. 	Therefore, the existence results for the case $\theta_{i}>0$ is quite different from that for the case $\theta_{i}=0$, $i=1,2$.
   		
   		On the other hand,	as in \cite{Zou 2012,Sirakov 2007}, we will use the Nehari manifold approach to	obtain the  constraint critical points on $\mathcal{N}$. But   the constraint critical point on $\mathcal{N}$ is not always a free critical point on $\mathcal{H}$.  For $\theta_{i}=0$,   the Nehari set is a natural constraint if  $-\infty<\beta<\sqrt{\mu_1\mu_2}$, as shown in \cite[Proposition 1.1]{Sirakov 2007}.
   		As we will see, the case $\theta_{i}>0$ is different.  For the case $\beta>0$ small, we prove that the Nehari set $\mathcal{N}$ is a natural constraint, see Proposition \ref{Lemma natural constraint} ahead. However, for the case $\beta<0$, because the sign of the logarithmic term is uncertain,  the Nehari set $\mathcal{N}$ may not be  a $C^1$-manifold, and  we cannot use the Lagrange multipliers rule to show that  the Nehari set is a natural constraint.  To overcome these difficulties, we narrow the range of $\beta$ to    $ \sbr{-\beta_0,0}$ and show that we can find the free critical points on a special set, see Proposition \ref{Lemma natural constraint2}  and Lemma \ref{Lemma natural constraint2} ahead.

   	}
   \end{remark}

   \begin{remark}
   	{\rm
   		Because of the lack of compactness of Sobolev embedding  $H_0^1(\Omega) \hookrightarrow L^{4}(\Omega)$,  it  is  difficult to prove the existence  of fully nontrivial solutions to the system \eqref{System1}. First we consider   the two cases   $\beta>0$ small and $\beta<0$. When $\theta_{i}=0$, Chen and Zou \cite[Lemma 5.1]{Zou 2012} obtained an existence result by  comparing the least energy level to \eqref{System1} with that
   		of limit system (see \eqref{System2} ahead) and scalar equations. However, the  presence of the logarithmic terms bring a lot of difficulties and make the comparison much more complicated,   so we need to  make more careful calculations, see Proposition \ref{Lemma energy estimate-1} and \ref{Lemma energy estimate-2} for details.
   		
   		For the case $\beta>0$ large,  we   follow the strategies in  \cite{Zou 2012} and  use  the mountain pass argument to prove (3) of Theorem \ref{Theorem-1}. One of the most important steps in their approach  is to compare the least energy level with that of Br\'ezis-Nirenberg  problem, as shown in \cite[Lemma 5.3]{Zou 2012}.  This comparison relies on a lower bound on the energy of Br\'ezis-Nirenberg  problem.  Inspired by that, we establish a lower bound on the energy of the single equation \eqref{Equation1}, as presented in Proposition \ref{Lemma bound of single energy}. Therefore,  we are able to establish an energy estimate comparing the least energy level $\mathcal{C}_{\mathcal{N}}$ with  that of the single equation \eqref{Equation1} for   $\beta>0$ large, see Proposition \ref{Lemma energy estimate-3} for details.   }
   \end{remark}

   \begin{remark}
   	{\rm
   		We mention that $\beta_0$, $\beta_1$ and  $\beta_2$  depend   on   $\lambda_{i}$, $\mu_i$, $\theta_i$ for $i=1,2$ and the accurate definitions of  $\beta_0$ and $\beta_1$  are given in   \eqref{Defi of beta0} and  	\eqref{Defi of beta1}, respectively.  	However, we are uncertain whether the ranges for $\beta$ in Theorem  \ref{Theorem-1} are optimal. This is  an open and interesting problem.
   	}
   \end{remark}

  \vskip 0.1in
    Theorem \ref{Theorem-1}  established  the existence   of positive solution for the case    $\theta_{i}>0$ for $i=1,2$. A natural question is whether or not the system \eqref{System1} has positive  solution (or nonnegative solution) when  both $\theta_{1}$ and $\theta_{2}$ are negative.  The following results Theorem \ref{Theorem-2} \ref{least ene solu} and  \ref{Theorem-5} give   affirmative answers to this question. Different from the case $\theta_{i}>0$,  the situations become thorny when $\theta_{i}<0$, $i=1,2$, since it is not easy to obtain the  Palais-Smale sequence in $\mathcal{N}$ . 	 The negative $\theta_1$ and $\theta_2$ change the geometry strucure of $\mathcal{L}$. More precisely, $(0,0)$ is  no longer a local minima, and we can find positive $(u,v) \in \mathcal{H}$ which is a local minima with negative energy level.
    
      In what follows, we consider the following  special sets,
     \begin{align}
     	A_1 := \{(\lambda_{1},\mu_{1},\theta_{1};\lambda_{2},\mu_{2},\theta_{2}):&\lambda_{1},\lambda_{2} \in [0,\lambda_{1}(\Omega)), \mu_{1},\mu_2 > 0, \theta_{1},\theta_{2} < 0, \\
     	&\frac{(\min\{\lambda_1(\Omega)-\lambda_1,\lambda_1(\Omega)-\lambda_2\})^2}{\lambda_1(\Omega)^2\max\{\mu_1,\mu_2\}}S^2 + 2(\theta_1+\theta_2)|\Omega| > 0\},
     \end{align}
     \begin{align}
     	A_2 := \{(\lambda_{1},\mu_{1},\theta_{1};\lambda_{2},\mu_{2},\theta_{2}):&\lambda_{1} \in [0,\lambda_{1}(\Omega)), \lambda_{2} \in \mathbb{R}, \mu_{1},\mu_2 > 0, \theta_{1},\theta_{2} < 0, \\
     	&\frac{(\lambda_1(\Omega)-\lambda_1)^2}{\lambda_1(\Omega)^2\max\{\mu_1,\mu_2\}}S^2 +  2(\theta_1+\theta_2e^{-\frac{\lambda_{2}}{\theta_{2}}})|\Omega| > 0\},
     \end{align}
     \begin{align}
     	A_3 := \{(\lambda_{1},\mu_{1},\theta_{1};\lambda_{2},\mu_{2},\theta_{2}):&\lambda_{1}, \lambda_{2} \in \mathbb{R}, \mu_{1},\mu_2 > 0, \theta_{1},\theta_{2} < 0, \\
     	&\frac{1}{\max\{\mu_1,\mu_2\}}S^2 + 2(\theta_1e^{-\frac{\lambda_{1}}{\theta_{1}}}+\theta_2e^{-\frac{\lambda_{2}}{\theta_{2}}})|\Omega| > 0\}.
     \end{align}
Here,  $\mathcal{S}$ denotes the Sobolev best constant of $\mathcal{D}^{1,2}(\R^4)\hookrightarrow L^4(\R^4)$,
\begin{equation}\label{sobolev constant}
	\mathcal{S}=\inf_{u \in \mathcal{D}^{1,2}(\R^4) \setminus \lbr{0} } \cfrac{\int_{\R^4} |\nabla u|^2 }{\sbr{\int_{\R^4} |u|^4}^{\frac{1}{2}}},
\end{equation}
where $\mathcal{D}^{1,2}(\R^4)=\lbr{u\in L^2(\R^4): |\nabla u| \in L^2(\R^4)}$ with norm $\left\|u \right\|_{\mathcal{D}^{1,2}}:=\sbr{\int_{\R^4}|\nabla u|^2 }^{\frac{1}{2}} $.
\medbreak
    Then the existence result   in this aspect  can be stated as follows.
     \begin{theorem} \label{Theorem-2}
     	Define
     	$$
     	\mathcal{C}_\rho := \inf_{(u,v)\in B_\rho}\mathcal{L}(u,v),
     	$$
     	where $B_r := \{(u,v)\in \mathcal{H}: \sqrt{|\nabla u|_2^2+|\nabla v|_2^2} < r\}$ and $\rho$ will be given by Lemma \ref{lm}. When $\beta < 0$, we assume that one of the following holds:
     	\begin{itemize}
     		\item[(i)] $(\lambda_1,\mu_1,\theta_1; \lambda_2,\mu_2,\theta_2) \in A_1$,
     		\item[(ii)] $(\lambda_1,\mu_1,\theta_1; \lambda_2,\mu_2,\theta_2) \in A_2$,
     		\item[(iii)] $(\lambda_1,\mu_1,\theta_1; \lambda_2,\mu_2,\theta_2) \in A_3$;
     	\end{itemize}
     	when $\beta > 0$, we assume that there exists $\epsilon > 0$ such that one of the following holds:
     	\begin{itemize}
     		\item[(iv)] $(\lambda_1,\mu_1+\beta\epsilon,\theta_1; \lambda_2,\mu_2+\frac{\beta}{\epsilon},\theta_2) \in A_1$,
     		\item[(v)] $(\lambda_1,\mu_1+\beta\epsilon,\theta_1; \lambda_2,\mu_2+\frac{\beta}{\epsilon},\theta_2) \in A_2$,
     		\item[(vi)] $(\lambda_1,\mu_1+\beta\epsilon,\theta_1; \lambda_2,\mu_2+\frac{\beta}{\epsilon},\theta_2) \in A_3$.
     	\end{itemize}
     Then system \eqref{System1} has a positive solution $(\tilde{u},\tilde{v})$ such that $\mathcal{L}(\tilde{u},\tilde{v}) = \mathcal{C}_\rho<0$, which is a local minima.
     \end{theorem}

     Then we will show the existence of a nonnegative solution with negative energy level reaching the minimum among all critical points.

     \begin{theorem} \label{least ene solu}
     		Define
     	\begin{equation}\label{Defi of C_K}
     			\mathcal{C}_K := \inf_{(u,v)\in K}\mathcal{L}(u,v),
     	\end{equation}
     	where
     	$$K = \{(u,v) \in \mathcal{H}: \mathcal{L}'(u,v) = 0\}.$$
     	Under the hypotheses of Theorem \ref{Theorem-2}. We further assume that
     	\begin{equation}
     		2\min\{\theta_1,\theta_2\} \geq -\lambda_{1}(\Omega) ~~  \text{ or }~~\beta > 0~~\text{ or }~~ \beta \in (-\sqrt{\mu_1\mu_2},0).
     	\end{equation}
Then the system \eqref{System1} possesses a nonnegative solution $(\hat{u},\hat{v}) \neq (0,0)$ such that $\mathcal{L}(\hat{u},\hat{v}) = \mathcal{C}_K<0$.
     \end{theorem}

     \begin{remark}
     	{\rm (1) It is an interesting question whether $\mathcal{C}_\rho = \mathcal{C}_K$ or not.
     	
     	(2) In Theorem \ref{least ene solu}, we don't know that $(u,v)$ is a least energy solution or not since we can not prove that $(u,v)$ is fully nontrivial. We leave it open in this paper.
     	
     	
     	(3) In Appendix \ref{single equation}, we address the single equation \eqref{Equation1} and   prove the existence of a positive local minima and the positive least energy solution with negative energy level, this improves the result of \cite[Theorem 1.3]{Deng-He-Pan=Arxiv=2022}.}
     \end{remark}

 The following theorem shows that there exists another nonnegative solution of system \eqref{System1}.
	\begin{theorem}\label{Theorem-5}

Under the hypotheses of Theorem \ref{Theorem-2}. Assume that  $\beta\in \sbr{-\infty,0}\cup \sbr{ 0,\min\lbr{\mu_1,\mu_2}}  \cup  \sbr{\max\lbr{\mu_1,\mu_2},+\infty} $, and  further
	\begin{equation}
		\frac{32 e^{{\lambda_{i}}/{\theta_{i}}}}{\mu_{i}R_{\max}^2}<1, \text{ with } \ \ R_{\max}:=\sup\lbr{r>0:\exists x\in \Omega, \text{ s.t. } B_r(x)\subset \Omega}, ~i=1,2.
	\end{equation}
	Then  the system \eqref{System1} possesses a nonnegative  solution $(\bar{u },\bar{v})\neq (0,0)$.
\end{theorem}
\begin{remark}
	{ \rm Inspired by Deng et al. \cite{Deng-He-Pan=Arxiv=2022}, we   use the mountain pass argument to obtain a bounded Palais-Smale sequence and to establish the corresponding energy estimates (see Proposition \ref{Lemma energy estimate-4} and \ref{Lemma energy estimate-5} ahead)  to show that the weak limit of  Palais-Smale sequence is a nonnegative solution   $(\bar{u },\bar{v})\neq (0,0)$. However, this solution may not be of the mountain pass type.    Define
		\begin{equation} \label{defi of B neqgative}
			\mathcal{C}_M:=\inf_{\gamma\in\Gamma}\max_{t\in \mbr{0,1}} \mathcal{L}(\gamma(t)),
		\end{equation}
		where  
		$$\Gamma := \{\gamma \in C([0,1],\mathcal{H}): \gamma(0) = (0,0), \mathcal{L}(\gamma(1)) <\mathcal{L}(\tilde{u},\tilde{v})<0\},$$ $(\tilde{u},\tilde{v})$ is the local minima given by Theorem \ref{Theorem-2}.
		 Here, we conjecture that
		 
	{\bf Conjecture 1:} System \eqref{System1} possesses a positive mountain pass solution at level $\mathcal{C}_M > 0$.
	
	If one   proves the following estimate:
\begin{equation}\label{C2}
		\mathcal{C}_M< \min\{\mathcal{C}_K + \frac14\mu_1^{-1}S^2, \mathcal{C}_K + \frac14\mu_2^{-1}S^2, \tilde{c}_{M,1}, \tilde{c}_{M,2}\},
\end{equation}
	where $\mathcal{C}_K $ is defined in \eqref{Defi of C_K} and  $\tilde{c}_{M,i}$ is defined in Appendix \ref{single equation} for equation \eqref{equation of u} with $(\lambda,\mu,\theta) = (\lambda_i,\mu_i,\theta_i), i = 1,2$, then Conjecture 1 holds true.
}
\end{remark}

For the nonexistence of positive solutions for the system \eqref{System1}, we have the following results. 	  We point out that  Theorem \ref{Theorem-3} deals with the case where   $\theta_{1}$  or $\theta_2$ is negative, while Theorem \ref{Theorem-4} handles the case where $\theta_{1}$ and $\theta_{2}$ have opposite signs. 
	\begin{theorem}\label{Theorem-3}
	Assume that $(\lambda_{1},\mu_1,\theta_{1}) \in \varSigma_2$ or $(\lambda_{2},\mu_2,\theta_{2}) \in \varSigma_2$, where
		\begin{equation}
			\varSigma_2:=\lbr{(\lambda,\mu,\theta): \theta<0,  \mu > 0, \text{ and } |\theta|+\theta \log |\theta|-\theta\log \mu+\lambda-\lambda_{1}(\Omega)\geq 0}.
		\end{equation}
		If $\beta>0$, then the system \eqref{System1}  has no positive solutions.
	\end{theorem}

	\begin{theorem} \label{Theorem-4}
		Assume that  $\mu_1<\mu_2$ and $\mu_1< \beta< \mu_2$. If $\theta_{2}<0< \theta_1$  and  $$-\theta_{1} \log \sbr{\frac{\theta_{1}}{\beta-\mu_{1}}}+\theta_{2} \log \sbr{\frac{\theta_{2}}{\beta-\mu_{2}}}+\theta_1-\theta_2 +\lambda_{2}-\lambda_{1}>0,$$ then   the system \eqref{System1}  has no positive solutions.
	\end{theorem}
	
	\begin{remark}
		{\rm If we assume that  $\mu_2<\mu_1$ and $\mu_2< \beta< \mu_1$, then the assumptions  are modified as follows:
			\begin{equation}
				\theta_{1}<0< \theta_2,  \text{ and }~-\theta_{1} \log \sbr{\frac{\theta_{1}}{\beta-\mu_{1}}}+\theta_{2} \log \sbr{\frac{\theta_{2}}{\beta-\mu_{2}}}+\theta_1-\theta_2 +\lambda_{2}-\lambda_{1}<0.
		\end{equation}}
		
	\end{remark}
	%
	\vskip 0.1in
	Before stating some preliminary results, we introduce some  notations and give the outline of our paper. 

Throughout this paper, we denote the norm of $L^p(\Omega)$  by $|\cdot|_p$ for  $1\leq p \leq \infty$.  We use ``$\to$'' and ``$\rightharpoonup$'' to denote the strong convergence and weak convergence in corresponding space respectively. The capital letter $C$ will appear as a constant which may vary from line to line, and $C_1$, $C_2$, $C_3$ are  prescribed constants. In Section \ref{Sect2}, we will give some preliminary results for the proof of Theorem \ref{Theorem-1} and complete its proof in Section \ref{Sect3}. We will prove Theorems \ref{Theorem-2}, \ref{least ene solu}  and \ref{Theorem-5} in Sections \ref{Sect4}, \ref{sect l-e-s} and \ref{Sect6}, respectively. Finally,  we show the proofs of nonexistence results  Theorem \ref{Theorem-3} and \ref{Theorem-4} in Section \ref{Sect5}.

	\vskip 0.1in
	

	
	
	\section{Preliminary results} \label{Sect2}
	In this section, we will introduce some preliminary results for $(\lambda_{i},\mu_i,\theta_i) \in$ $\varSigma_1$ for $i=1,2$.
	From now on, we assume that $\beta \in \sbr{-\infty,0}\cup \sbr{0,\min \lbr{\mu_1,\mu_2}}$.
	\subsection{The Nehari set}\label{Sect2.1}
	Let
	\begin{equation} \label{Defi of beta1}
		\beta_1:=\min \lbr{  \mu_1,\mu_2 , \frac{\sqrt{\mu_1}}{2\sqrt{2\sbr{\mu_1^{-1}+\mu_2^{-1}}}}, \frac{\sqrt{\mu_2}}{2\sqrt{2\sbr{\mu_1^{-1}+\mu_2^{-1}}}} } .
	\end{equation}
	Now   we   establish both lower and upper uniform estimates on the $L^4$-norms of elements in the Nehari set that fall below a certain energy level.
	\begin{lemma} \label{Lemma L^4 norm bounded}
		Let $\beta\in  \sbr{-\sqrt{\mu_1\mu_2},0} \cup \sbr{0,\beta_1}$.  Then  there exist $C_2>C_1>0$,  such that for any $(u,v)\in \mathcal{N}$ with $\mathcal{L}(u,v)\leq \frac{1}{2}(\mu_1^{-1}+\mu_2^{-1})\mathcal{S}^2$, there holds
		\begin{equation}
			C_1\leq |u^+|_4^4,|v^+|_4^4 \leq C_2.
		\end{equation}
		Here, $C_1$, $C_2$  depend only on   $\lambda_{i}$, $\mu_i$, $\theta_i$ for $i=1,2$.
	\end{lemma}
	
	\begin{proof}
		Take any $(u,v)\in \mathcal{N}$ with $\mathcal{L}(u,v)\leq \frac{1}{2}(\mu_1^{-1}+\mu_2^{-1})\mathcal{S}^2$.
		Firstly, when $-\sqrt{\mu_1\mu_2}<\beta<0$,  since the inequality  $s^2\log s^2\leq e^{-1} s^4$ holds, we have
		\begin{equation}
			\begin{aligned}
				\mathcal{S}|u^+|^2_4\leq |\nabla u|_2^2 & =\lambda_{1}|u^+|_2^2+\mu_1 |u^+|_4^4+\theta_1\int_{\Omega}(u^+)^2\log (u^+)^2+\beta|u^+v^+|_2^2 \\
				&\leq \mu_1 |u^+|_4^4 + \theta_1\int_{\Omega}( u^+)^2\log (e^{\frac{\lambda_{1}}{\theta_1}}(u^+)^2)\\
				&\leq \sbr{\mu_1+\theta_1e^{\frac{\lambda_{1}}{\theta_1}-1}}|u^+|_4^4.
			\end{aligned}
		\end{equation}
		Therefore, we have
		$|u^+|_4^4 \geq \mathcal{S}^2\sbr{\mu_1+\theta_1e^{\frac{\lambda_{1}}{\theta_1}-1}}^{-2}$. Similarly, we have $|v^+|_4^4 \geq \mathcal{S}^2\sbr{\mu_2+\theta_2e^{\frac{\lambda_{2}}{\theta_2}-1}}^{-2}$.
		
		On the other hand,   we have
		\begin{equation} \label{f23}
			\begin{aligned}
				\frac{1}{2}(\mu_1^{-1}+\mu_2^{-1})\mathcal{S}^2 \geq \mathcal{L}(u,v)&=\mathcal{L}(u,v)-\frac{1}{4}\mathcal{L}^\prime(u,v)(u,v)\\
				&= \frac{1}{4}|\nabla u|_2^2 -\frac{\theta_1}{4}\int_{\Omega}(u^+)^2 \log \sbr{e^{\frac{\lambda_{1}}{\theta_1}} (u^+)^2}+\frac{\theta_1}{2}|u^+|_2^2 \\
				& +\frac{1}{4}|\nabla v|_2^2 -\frac{\theta_2}{4}\int_{\Omega}(v^+)^2 \log \sbr{e^{\frac{\lambda_{2}}{\theta_2}} (v^+)^2}+\frac{\theta_2}{2}|v^+|_2^2 .
			\end{aligned}
		\end{equation}
		Moreover,
		\begin{equation} \label{f1}
			\begin{aligned}
				\frac{1}{2}(\mu_1^{-1}+\mu_2^{-1})\mathcal{S}^2 \geq	\mathcal{L}(u,v)&=\mathcal{L}(u,v)-\frac{1}{2}\mathcal{L}^\prime(u,v)(u,v)\\&=\frac{1}{4}\sbr{\mu_1|u^+|_4^4+\mu_2|v^+|_4^4+2\beta |u^+v^+|_2^2}+\frac{\theta_1}{2}|u^+|_2^2+\frac{\theta_2}{2}|v^+|_2^2.
			\end{aligned}
		\end{equation}
		Since  $\beta>-\sqrt{\mu_1\mu_2}$, there exists some positive constants $c_0,C_0>0$ such that
		\begin{equation}\label{f63}
			c_0\sbr{|u^+|_4^4+|v^+|_4^4}\leq \mu_1|u^+|_4^4+\mu_2|v^+|_4^4+2\beta |u^+v^+|_2^2\leq C_0\sbr{|u^+|_4^4+|v^+|_4^4}.
		\end{equation}
		Therefore, we have
		\begin{equation} \label{f2}
			|u^+|_2^2 \leq  (\mu_1^{-1}+\mu_2^{-1})\theta_1^{-1}\mathcal{S}^2 \text{ and } |v^+|_2^2 \leq  (\mu_1^{-1}+\mu_2^{-1})\theta_2^{-1}\mathcal{S}^2.
		\end{equation}
		Recalling the following useful inequality (see \cite{W.Shuai=Nonlineariyu=2019} or \cite[Theorem 8.14]{Lieb=2001})
		\begin{equation}
			\int_{\Omega} u^2 \log u^2 \leq \frac{a}{\pi}|\nabla u|_2^2+\sbr{\log |u|_2^2-N(1+\log a)}|u|_2^2  \ \ \text{ for } u\in H_0^1(\Omega) \  \text{ and } \  a>0,
		\end{equation}
		where $N=4$  is the dimension of $\Omega$. Let $w^+=e^{\frac{\lambda_{1}}{2\theta_1}} u^+$ and $z^+=e^{\frac{\lambda_{2}}{2\theta_2}} v^+$.  Since $|s^2\log s^2| \leq Cs^{2-\tau}+Cs^{2+\tau}$ for  any $\tau\in \sbr{0,1}$, we have
		\begin{equation}\label{f8}
			\begin{aligned}
				\frac{1}{4}\sbr{|\nabla u|_2^2+|\nabla v|_2^2} &\leq \frac{\mu_1^{-1}+\mu_2^{-1}}{2}\mathcal{S}^2+ \frac{\theta_1}{4}e^{-\frac{\lambda_{1}}{\theta_1}}\int_{\Omega}(w^+)^2 \log (w^+)^2+\frac{\theta_2}{4}e^{-\frac{\lambda_{2}}{\theta_2}}\int_{\Omega}(z^+)^2 \log  (z^+)^2\\
				&	\leq \frac{\mu_1^{-1}+\mu_2^{-1}}{2}\mathcal{S}^2+\frac{\theta_1}{4}e^{-\frac{\lambda_{1}}{\theta_1}}\mbr{\frac{a}{\pi}|\nabla w|_2^2+|w^+|_2^2\log |w^+|_2^2-N(1+\log a)|w^+|_2^2 }\\
				&+\frac{\theta_2}{4}e^{-\frac{\lambda_{2}}{\theta_2}}\mbr{\frac{a}{\pi}|\nabla z|_2^2+|z^+|_2^2\log |z^+|_2^2-N(1+\log a)|z^+|_2^2 }\\
				&\leq \frac{\mu_1^{-1}+\mu_2^{-1}}{2}\mathcal{S}^2+ \frac{1}{8}|\nabla u|_2^2+ C|u^+|_2^{2-\tau} +
				C|u^+|_2^{2+\tau} +C|u^+|_2^2\\
				&+	\frac{1}{8}|\nabla v|_2^2+ C|v^+|_2^{2-\tau} +
				C|v^+|_2^{2+\tau} +C|v^+|_2^2,
			\end{aligned}
		\end{equation}
		where we fix $a>0$ with $\frac{a}{\pi}\theta_1<\frac{1}{2}$ and  $\frac{a}{\pi}\theta_2<\frac{1}{2}$. Therefore, combining this with \eqref{f2}, we can see that there exists $C_2>0$, such that
		\begin{equation}
			|u^+|_4^4\leq S^{-2} |\nabla u|_2^4\leq C_2.
		\end{equation}
		Similarly, we have $|v^+|_4^4\leq C_2$.
		
		Secondly, when $\beta\in (0,\beta_1)$, since $\theta_i>0$ for $i=1,2$, we can see from \eqref{f1}that
		\begin{equation}
			|u^+|^4_4\leq 2\mu_1^{-1}( \mu_1^{-1}+ \mu_2^{-1})\mathcal{S}^2 \ \ \text{ and } \ \ 	|v^+|^4_4\leq 2\mu_2^{-1}( \mu_1^{-1}+ \mu_2^{-1})\mathcal{S}^2 .
		\end{equation}
		Since $(u,v)\in \mathcal{N}$, we have
		\begin{equation}
			\begin{aligned}
				\mathcal{S}|u^+|_4^2\leq |\nabla u|_2^2&=\lambda_{1} |u^+|_2^2+\mu_1|u^+|_4^4+\beta|u^+v^+|_2^2+\theta_1\int_{\Omega} (u^+)^2\log (u^+)^2\\
				& \leq \mu_1|u^+|_4^4+\beta |u^+|_4^2|v^+|_4^2 +\theta_1\int_{\Omega}  (u^+)^2\log (e^{\frac{\lambda_{1}}{\theta_1}}(u^+)^2)\\
				&\leq \sbr{\mu_1+\theta_1e^{\frac{\lambda_{1}}{\theta_1}-1}}|u^+|_4^4+ \beta \sqrt{ 2\mu_2^{-1}\sbr{\mu_1^{-1}+ \mu_2^{-1}}}\mathcal{S}|u^+|_4^2\\
				&\leq\sbr{\mu_1+\theta_1e^{\frac{\lambda_{1}}{\theta_1}-1}}|u^+|_4^4+ \frac{1}{2}	\mathcal{S}|u^+|_4^2.
			\end{aligned}
		\end{equation}
		which uses the assumption $\beta\in \sbr{0,\beta_1}$. Therefore,
		\begin{equation}
			|u^+|_4^4\geq \frac{\mathcal{S}^2}{4\sbr{\mu_1+\theta_1e^{\frac{\lambda_{1}}{\theta_1}-1}}^2}.
		\end{equation}
		Similarly, we have
		\begin{equation}
			|v^+|_4^4\geq \frac{\mathcal{S}^2}{4\sbr{\mu_2+\theta_1e^{\frac{\lambda_{2}}{\theta_2}-1}}^2}.
		\end{equation}
		This completes the proof.
	\end{proof}

	%
	\
	The following proposition shows that the set $\mathcal{N}$  is a natural constraint when $\beta \in \sbr{0,\beta_1}$.
	\begin{proposition} \label{Lemma natural constraint}
		Assume that $\beta \in \sbr{0,\beta_1}$ and  the energy level $\mathcal{C}_{\mathcal{N}}$ is achieved by $(u,v)  \in \mathcal{N}$. Then $(u,v)$ is a critical point of the functional  $\mathcal{L}$.
	\end{proposition}
	
	\begin{proof}
		Let
		\begin{equation}
			\mathcal{G}_1(u,v):=\mathcal{L}^\prime(u,v)(u,0) \text{ and }\mathcal{G}_2(u,v):=\mathcal{L}^\prime(u,v)(0,v).
		\end{equation}
		Take $(u,v) \in \mathcal{N}$. Since the parameters $\mu_1,\beta,\theta_1>0$, we deduce from  a direct calculation that
		$$	\mathcal{G}_1^\prime(u,v)(u,v)= -2\sbr{\mu_1\int_{\Omega}|u^+|^4+\int_{\Omega} \beta |u^+|^2|v^+|^2+\theta_1\int_{\Omega}|u^+|^2 }<0,$$  Similarly,  we have $	\mathcal{G}_2^\prime(u,v)(u,v)<0$. Therefore,   it follows that  $\mathcal{G}_i(u,v)$    defines,
		locally, a  $C^1$-manifold of codimension  1 in $\mathcal{H}$ for any  $(u,v) \in \mathcal{N}$  and  $i=1,2$.   Now we claim that  in a neighborhood of  $(u,v) \in \mathcal{N}  $, the set $\mathcal{N}$ is a $C^1$-manifold of codimension 2 in $\mathcal{H}$. For that purpose, we only need to show that
		$(	\mathcal{G}_1^\prime(u,v),	\mathcal{G}_2^\prime(u,v))$ is a surjective as a linear operator $\mathcal{H} \to \R^2$.
		Notice that
		\begin{equation}\label{f51}
			\mathcal{G}_1^\prime(u,v)(t_1u,t_2v)=-2\sbr{\mu_1\int_{\Omega}|u^+|^4+\theta_1\int_{\Omega}|u^+|^2}t_1-2\sbr{\int_{\Omega} \beta |u^+|^2|v^+|^2 }t_2,
		\end{equation}
		\begin{equation}\label{f52}
			\mathcal{G}_2^\prime(u,v)(t_1u,t_2v)=-2\sbr{\int_{\Omega} \beta |u^+|^2|v^+|^2 }t_1-2\sbr{\mu_2\int_{\Omega}|v^+|^4+\theta_2\int_{\Omega}|v^+|^2}t_2,
		\end{equation}	
		and the matrix
		\begin{equation}\label{matrix 2}
			\begin{pmatrix}
				\mu_1|u^+|_4^4+\theta_1|u^+|_2^2 & \beta|u^+v^+|_2^2 \\
				\beta|u^+v^+|_2^2  & \mu_2|v^+|_4^4 +\theta_2|v^+|_2^2
			\end{pmatrix}
		\end{equation}
		has strictly positive determinant. 	The positivity of the determinant follows from the fact $\theta_1,\theta_2>0$, $0<\beta < \beta_1\leq \min \lbr{\mu_1,\mu_2}$ and the H\"{o}lder's inequality.  Then for  any $s_1,s_2\in\R$,   there exist $t_{1},t_2 \in\R$ such that	$$(	\mathcal{G}_1^\prime(w,z),	\mathcal{G}_2^\prime(w,z))=(s_1,s_2),$$
		where $(w,z)=(t_1u,t_2v)$. Therefore, the claim is true.

		%
		%
		%
		%
		
		We suppose that $\mathcal{C}_{\mathcal{N}}$ is achieved  by $(u,v)  \in \mathcal{N}$. Since  the set $\mathcal{N}$ is a $C^1$-manifold of codimension 2   in a neighborhood of  $(u,v) \in \mathcal{N} $. Then by the Lagrange multipliers rule, there exist $L_1,L_2\in \R$, such that
		\begin{equation} \label{f25}
			\mathcal{L}^\prime (u,v)-L_1 \mathcal{G}_1^\prime(u,v)-L_2 \mathcal{G}_2^\prime(u,v)=0.
		\end{equation}
		Multiplying the above equation with $\sbr{u,0}$ and  $\sbr{0,v}$, we deduce from    $	\mathcal{L}^\prime (u,v)(u,0)=	\mathcal{L}^\prime (u,v)(0,v)=0$ that
		\begin{equation} \label{f26}
			L_1\sbr{\mu_1\int_{\Omega}|u^+|^4+\theta_1\int_{\Omega}|u^+|^2}+L_2 \int_{\Omega} \beta |u^+|^2|v^+|^2=0.
		\end{equation}
		\begin{equation} \label{f27}
			L_1\int_{\Omega} \beta |u^+|^2|v^+|^2+L_2 \sbr{\mu_2\int_{\Omega}|v^+|^4+\theta_2\int_{\Omega}|v^+|^2}=0.
		\end{equation}
		Since the system \eqref{f26}-\eqref{f27} has a strictly positive determinant (since $u,v\not \equiv 0$ on $\mathcal{N}$), by the Cramer's rule, the system has a unique solution $L_1=L_2=0$, which implies that $\mathcal{L}^\prime (u,v)=0$ and   the Nehari manifold $\mathcal{N}$  is a  natural constraint in $\mathcal{H}$.
	\end{proof}
	
	We notice that the inequalities  $\mathcal{G}_i(u,v)<0$ for $i=1,2$ play an important role in the proof of Proposition \ref{Lemma natural constraint}.  However, it seems difficult to obtain such  inequalities when  $\beta \in \sbr{-\infty,0}$. Then the Nehari set may not be  a $C^1$-manifold, and  we can not use the Lagrange multipliers rule. To overcome this difficulty, we consider the matrix
	\begin{equation}\label{matrix 1}
		M(u,v)=\begin{pmatrix}
			M_{11}(u,v)  &  M_{12}(u,v) \\
			M_{21}(u,v)  &  M_{22}(u,v)
		\end{pmatrix}:=\begin{pmatrix}
			\mu_1|u^+|_4^4+\theta_1|u^+|_2^2 & \beta|u^+v^+|_2^2 \\
			\beta|u^+v^+|_2^2  & \mu_2|v^+|_4^4 +\theta_2|v^+|_2^2
		\end{pmatrix},
	\end{equation}
	and the set
	\begin{equation}
		\begin{aligned}
			\mathcal{Q}&:=\lbr{(u,v)\in \mathcal{H}: \text{ the matrix } M(u,v) \text{ is strictly diagonally dominant}}\\
			& =\lbr{(u,v)\in \mathcal{H}:  M_{11}(u,v)-|M_{12}(u,v)|>0, M_{22}(u,v)-|M_{21}(u,v)|>0 }.
		\end{aligned}
	\end{equation}
	Then we show that the set $\mathcal{N}\cap \mathcal{Q}$  is a natural constraint when $\beta<0$.
	
	\begin{proposition} \label{Lemma natural constraint2}
		Assume that $\beta <0$ and  the energy level $\mathcal{C}_{\mathcal{N}}$ is achieved by $(u,v)  \in \mathcal{N}\cap \mathcal{Q}$. Then $(u,v)$ is a critical point of the functional  $\mathcal{L}$. That is, the set $\mathcal{N}\cap \mathcal{Q}$  is a natural constraint.
	\end{proposition}
	\begin{proof}
		Take any $(u,v) \in \mathcal{N}\cap \mathcal{Q}$.  Since  the matrix  $M(u,v)$  is strictly diagonally dominant and $\beta<0$, we have   $\mathcal{G}_i^\prime(u,v)(u,v)<0$ for $i=1,2$, Therefore, $\mathcal{G}_i(u,v)$    defines,
		locally, a  $C^1$-manifold of codimension  1 in $\mathcal{H}$ for any  $(u,v) \in \mathcal{N}\cap \mathcal{Q}$   and  $i=1,2$.  Since  $(u,v)\in \mathcal{Q}$, the matrix \eqref{matrix 2} is positive definite. Then,
		similar to the proof of Proposition \ref{Lemma natural constraint}, we can show that  the set $\mathcal{N}$ is a $C^1$-manifold of codimension 2 in $\mathcal{H}$ in  a neighborhood of  $(u,v) \in \mathcal{N}\cap \mathcal{Q} $.
		
		Now we suppose that $\mathcal{C}_{\mathcal{N}}$ is achieved  by $(u,v)  \in \mathcal{N}\cap \mathcal{Q}$. By the Sobolev embedding theorem, the set $ \mathcal{N}\cap \mathcal{Q}$ is an open set of $\mathcal{N}$ in the topology of $\mathcal{H}$. Thus $(u,v)$ is an inner critical point of $\mathcal{L}$ in an open subset of $\mathcal{N}$, and in particular it is a constrained critical point of $\mathcal{L}$ on  $\mathcal{N}$. Then  the conclusion follows word by word the ones of Proposition \ref{Lemma natural constraint}, that is, the set $\mathcal{N}\cap \mathcal{Q}$  is a natural constraint.
	\end{proof}
	
	Let
	\begin{equation} \label{Defi of beta0}
		\beta_0:=\min \lbr{ \sqrt{\mu_1\mu_2},\mu_1 \sqrt{\frac{C_1}{C_2}},\mu_2 \sqrt{\frac{C_1}{C_2}}}\leq \min \lbr{\mu_1,\mu_2}.
	\end{equation}
	Then the following lemma shows that we can obtain the critical points of  $\mathcal{L}$ on  $\mathcal{H} $ by considering the    critical points of  $\mathcal{L}$ on $\mathcal{N}\cap  \lbr{(u,v)\in \mathcal{H}: \mathcal{L}(u,v)\leq \frac{1}{2}(\mu_1^{-1}+\mu_2^{-1})\mathcal{S}^2}$.
	
	\begin{lemma} \label{Lemma natural constraint3}
		Let $\beta\in   \sbr{-\beta_0,0} $, then we have
		\begin{equation}
			\mathcal{N}\cap  \lbr{(u,v)\in \mathcal{H}: \mathcal{L}(u,v)\leq \frac{1}{2}(\mu_1^{-1}+\mu_2^{-1})\mathcal{S}^2}\subset  	 \mathcal{N}\cap  \mathcal{Q}.
		\end{equation}
	\end{lemma}
	\begin{proof}
		This proof   comes  directly from  Lemma \ref{Lemma L^4 norm bounded} and the definition of $\mathcal{Q}$ by using the H\"{o}lder's inequality.
	\end{proof}

	\subsection{Energy estimates}\label{Sect2.2}
	
	Consider the  Br\'ezis-Nirenberg problem with logarithmic perturbation
	\begin{equation}
		-\Delta u=\lambda_i u+\mu_i |u|^2u+\theta_i u \log u^2 \quad \text{ in }\Omega, ~i=1,2,
	\end{equation}
	where $\lambda_{i}\in \R$, $\mu_i, \theta_{i}>0$. As in \cite{Deng-He-Pan=Arxiv=2022}, we define the associated modified energy functional
	\begin{equation}
		\mathcal{L}_i(u)=\frac{1}{2} \int_{\Omega}|\nabla u|^2-\frac{\lambda_i}{2} \int_{\Omega}|u^+|^2-\frac{\mu_i}{4}\int_{\Omega} |u^+|^4-\frac{\theta_i}{2} \int_{\Omega} (u^+)^2\sbr{\log (u^+)^2-1},
	\end{equation}
	and the level
	\begin{equation}
		\mathcal{C}_{\theta_i} =\inf_{u \in \mathcal{N}_i }	\mathcal{L}_i(u),
	\end{equation}
	where
	\begin{equation}
		\mathcal{N}_i=\lbr{u\in H_0^1(\Omega)\setminus \lbr{0}: \mathcal{L}_i^\prime(u)u=0 }.
	\end{equation}
	Then we have the following.
	\begin{proposition} \label{Lemma bound of single energy}
		For $i=1,2$,  we have $$ \cfrac{\mu_i\mathcal{S}^2}{4\sbr{\mu_i+\theta_i e^{\lambda_i/\theta_i-1}}^2}\leq\mathcal{C}_{\theta_i} < \frac{1}{4}\mu_i^{-1}\mathcal{S}^2.$$
	\end{proposition}
	\begin{proof}
		By   \cite[Lemma 3.5]{Deng-He-Pan=Arxiv=2022}  , we can easily see that $$\mathcal{C}_{\theta_i}<\frac{1}{4}\mu_i^{-1}\mathcal{S}^2, \quad \text{ for } i=1,2.$$
		On the other hand, for any $u\in \mathcal{N}_i$, by the inequality  $s^2\log s^2\leq e^{-1} s^4$ for all $s>0$, we have
		\begin{equation}
			\begin{aligned}
				\mathcal{S}|u^+|^2_4\leq |\nabla u|_2^2
				&= \mu_i |u^+|_4^4 + \theta_i\int_{\Omega}( u^+)^2\log (e^{\frac{\lambda_{i}}{\theta_i}}(u^+)^2)\\
				&\leq \sbr{\mu_i+\theta_i e^{\lambda_i/\theta_i-1}}|u^+|_4^4,
			\end{aligned}
		\end{equation}
		which implies that $|u^+|_4^4\geq \cfrac{\mathcal{S}^2}{\sbr{\mu_i+\theta_i e^{\lambda_i/\theta_i-1}}^2}$. So, we have
		\begin{equation}
			\mathcal{L}_i(u)=\mathcal{L}_i(u)-\frac{1}{2}\mathcal{L}^\prime_i(u)u=\frac{\mu_i}{4}\int_{\Omega} |u^+|^4+\frac{\theta_i}{2} \int_{\Omega} |u^+|^2 \geq \frac{\mu_i}{4}\int_{\Omega} |u^+|^4\geq  \cfrac{\mu_i\mathcal{S}^2}{4\sbr{\mu_i+\theta_i e^{\lambda_i/\theta_i-1}}^2},
		\end{equation}
		which implies that $$\mathcal{C}_{\theta_i}\geq \cfrac{\mu_i\mathcal{S}^2}{4\sbr{\mu_i+\theta_i e^{\lambda_i/\theta_i-1}}^2}~, \quad \text{ for } i=1,2.$$	This completes the proof.
	\end{proof}
	
	\begin{proposition} \label{Lemma energy estimate-1}
		Let $\beta \in  \sbr{-\beta_0,0}\cup \sbr{0,\beta_1}$. Then there holds
		\begin{equation}
			\mathcal{C}_{\mathcal{N}}< \min \lbr{ \mathcal{C}_{\theta_1}+\frac{1}{4}\mu_2^{-1}\mathcal{S}^2, \ \  \mathcal{C}_{\theta_2}+\frac{1}{4}\mu_1^{-1}\mathcal{S}^2,\ \ \frac{1}{4}\sbr{\mu_1^{-1}+\mu_2^{-1}}\mathcal{S}^2}.
		\end{equation}
	\end{proposition}
	
	\begin{proof}
		Without loss of generality, we prove that
		\begin{equation}
			\mathcal{C}_{\mathcal{N}}<\mathcal{C}_{\theta_1}+\frac{1}{4} \mu_2^{-1}\mathcal{S}^2.
		\end{equation}
		By \cite{Deng-He-Pan=Arxiv=2022}, the energy level $ \mathcal{C}_{\theta_1}$ can be achieved by positive solution $u_{\theta_1} $. Moreover, we have  $u_{\theta_1}\in C^2(\Omega) $ and $u_{\theta_1}\equiv 0$ on $\partial \Omega$. Then, there exists  a ball
		\begin{equation}
			B_{2R_0}(y_0) :=\lbr{x\in \Omega: |x-y_0|\leq 2R} \subset \Omega,
		\end{equation}
		satisfying
		\begin{equation}\label{f4}
			\Pi^2:=\max_{B_{2R_0}(y_0)} |u_{\theta_1}|^2\leq \frac{\theta_2}{2|\beta|}.
		\end{equation}

		Let $\xi\in C_0^\infty(\Omega)$ be the radial function, such that  $\xi(x)\equiv1$ for $0\leq |x-y_0|\leq R$, $0\leq \xi(x)\leq 1$ for $R\leq |x-y_0|\leq 2R$, $\xi(x)\equiv0$ for $ |x-y_0|\geq  2R$, where we take arbitrary $R<R_0$ such that $B_{2R}(y_0)\subset B_{2R_0}(y_0)$.  Take $v_\varepsilon(x)=\xi(x) U_{\varepsilon,y_0}(x)$, where
		\begin{equation} \label{limit system_1}
			U_{\varepsilon,y_0}(x)= \cfrac{2\sqrt{2}\varepsilon}{\varepsilon^2+|x-y_0|^2}.
		\end{equation}
		Then by \cite{Brezis-Nirenberg1983} or \cite[Lemma 1.46]{william=1996}, we obtain the following inequalities
		\begin{equation}\label{es-1}
			\begin{aligned}
				\int_{\Omega} |\nabla v_\varepsilon|^2 =\mathcal{S}^2+O(\varepsilon^2), \quad  	\int_{\Omega} | v_\varepsilon|^4 =\mathcal{S}^2+O(\varepsilon^4).
			\end{aligned}
		\end{equation}
		Also, by \cite[Lemma 3.4]{Deng-He-Pan=Arxiv=2022}, we have the following inequalities
		\begin{equation}\label{es-2}
			\int_{\Omega}|v_\varepsilon|^2= 8\omega_4\varepsilon^2|\log \varepsilon| +O(\varepsilon^2),
		\end{equation}
		\begin{equation}\label{es-3}
			\int_{\Omega} v_\varepsilon^2 \log v_\varepsilon^2 \geq 8\log \sbr{\cfrac{8(\varepsilon^2+R^2)}{e(\varepsilon^2+4R^2)^2}}\omega_4 \varepsilon^2\log(\frac{1}{\varepsilon})+O(\varepsilon^2),
		\end{equation}
		and
		\begin{equation}\label{es-4}
			\int_{\Omega} v_\varepsilon^2 \log v_\varepsilon^2 \leq  8\log \sbr{\cfrac{8e(\varepsilon^2+4R^2)}{(\varepsilon^2+R^2)^2}}\omega_4 \varepsilon^2\log(\frac{1}{\varepsilon})+O(\varepsilon^2),
		\end{equation}
		where $\omega_4$ denotes the area of the unit sphere surface in $\R^4$. Moreover, by \eqref{f4} we have
		\begin{equation}\label{es-5}
			|\beta |\int_{\Omega} |u_{\theta_1}|^2|v_\varepsilon|^2=|\beta| \int_{B_{2R}(y_0) } |u_{\theta_1}|^2|v_\varepsilon|^2 \leq |\beta| \Pi^2 \int_{\Omega}|v_\varepsilon|^2 \leq \frac{\theta_2}{2}\int_{\Omega}|v_\varepsilon|^2=O(\varepsilon^2|\log \varepsilon|).
		\end{equation}
		
		Now we claim that there exists $s_{1,\varepsilon},s_{2,\varepsilon}>0$ such that $\sbr{s_{1,\varepsilon}u_{\theta_1},s_{2,\varepsilon}v_{\varepsilon}}\in \mathcal{N}$.
		For that purpose, we consider
		\begin{equation}
			\begin{aligned}
				F(t_1,t_2)&:=\mathcal{L}(t_1u_{\theta_1},t_2v_{\varepsilon})\\&=\frac{1}{2}t_1^2\int_{\Omega}|\nabla u_{\theta_1}|^2-\frac{\lambda_1}{2}t_1^2\int_{\Omega}|u_{\theta_1}|^2-\frac{\mu_1}{4}t_1^4 \int_{\Omega}|u_{\theta_1}|^4-\frac{\theta_1}{2}\int_{\Omega}\sbr{t_1u_{\theta_1}}^2(\log (t_1u_{\theta_1})^2-1)\\&+\frac{1}{2}t_2^2\int_{\Omega} |\nabla v_{\varepsilon}|^2-\frac{\lambda_2}{2}t_2^2\int_{\Omega}|v_{\varepsilon}|^2-\frac{\mu_2}{4}t_2^4 \int_{\Omega}|v_{\varepsilon}|^4-\frac{\theta_2}{2}\int_{\Omega}(t_2v_{\varepsilon})^2(\log (t_2v_{\varepsilon})^2-1)\\&-\frac{\beta}{2}t_1^2t_2^2\int_{\Omega}|u_{\theta_1}|^2|v_{\varepsilon}|^2.
			\end{aligned}
		\end{equation}
		For $\varepsilon$ small enough,  we can see from \eqref{es-1} and  \eqref{es-2}  that  the matrix
		\begin{equation}
			\begin{pmatrix}
				\mu_1|u_{\theta_1}|_4^4 & \beta|u_{\theta_1}v_{\varepsilon}|_2^2 \\
				\beta|u_{\theta_1}v_{\varepsilon}|_2^2  & \mu_2|v_{\varepsilon}|_4^4
			\end{pmatrix}
		\end{equation}
		is   strictly diagonally dominant, so  it is  positive definite.  Therefore, there exists a constant $C>0$ such that
		\begin{equation}
			\frac{\mu_1}{4}t_1^4 \int_{\Omega}|u_{\theta_1}|^4+\frac{\beta}{2}t_1^2t_2^2\int_{\Omega}|u_{\theta_1}|^2|v_{\varepsilon}|^2+\frac{\mu_2}{4}t_2^4 \int_{\Omega}|v_{\varepsilon}|^4\geq C(t_1^4+t_2^4).
		\end{equation}
		Then
		\begin{equation}
			\begin{aligned}
				F(t_1,t_2)&\leq \frac{1}{2}t_1^2\int_{\Omega}|\nabla u_{\theta_1}|^2-\frac{\lambda_1}{2}t_1^2\int_{\Omega}|u_{\theta_1}|^2-Ct_1^4 -\frac{\theta_1}{2}\int_{\Omega}\sbr{t_1u_{\theta_1}}^2(\log (t_1u_{\theta_1})^2-1)\\&+\frac{1}{2}t_2^2\int_{\Omega} |\nabla v_{\varepsilon}|^2-\frac{\lambda_2}{2}t_2^2\int_{\Omega}|v_{\varepsilon}|^2-Ct_2^4 -\frac{\theta_2}{2}\int_{\Omega}(t_2v_{\varepsilon})^2(\log (t_2v_{\varepsilon})^2-1).
			\end{aligned}
		\end{equation}
		Since $\lim_{s \to +\infty}\frac{s^4}{s^2\log s^2} =+\infty$, we can see that $	F(t_1,t_2)  \to -\infty$, as $| (t_1,t_2)| \to +\infty$, where $| (t_1,t_2)| =\sqrt{t_1^2+t_2^2}$. Hence, there exists a global maximum point $(s_{1,\varepsilon},s_{2,\varepsilon})\in \overline{\sbr{\R^+}^2}$.
		
		Assume that $(s_{1,\varepsilon},s_{2,\varepsilon})\in  \partial \overline{\sbr{\R^+}^2}$. Without loss of generality, we assume that $s_{1,\varepsilon}=0$ and $s_{2,\varepsilon}\neq 0$. Since $\lim_{s \to 0+} \frac{s^p}{s^2\log s^2}=0																																																																																								$ for  any $p \geq 2 $, we have
		\begin{equation}
			\begin{aligned}
				F(t_1,s_{2,\varepsilon})-F(s_{1,\varepsilon},s_{2,\varepsilon})&=\frac{1}{2}t_1^2\int_{\Omega}|\nabla u_{\theta_1}|^2-\frac{\lambda_1}{2}t_1^2\int_{\Omega}|u_{\theta_1}|^2-\frac{\mu_1}{4}t_1^4 \int_{\Omega}|u_{\theta_1}|^4-\frac{\theta_1}{2}t_1^2\log t_1^2 \int_{\Omega}|u_{\theta_1}|^2\\&-\frac{\theta_1}{2}t_1^2\int_{\Omega}u_{\theta_1}^2(\log u_{\theta_1}^2-1)-\frac{\beta}{2}t_1^2s_{2,\varepsilon}^2\int_{\Omega}|u_{\theta_1}|^2|v_{\varepsilon}|^2>0
			\end{aligned}
		\end{equation}
		for $t_1$ small enough. This contradicts to the fact that $(s_{1,\varepsilon},s_{2,\varepsilon})$ is  a global maximum point in $\overline{\sbr{\R^+}^2}$. Therefore,  $(s_{1,\varepsilon},s_{2,\varepsilon}) \not \in  \partial \overline{\sbr{\R^+}^2}$ and it is a critical point of $F(t_1,t_2)$. Also, we can see that $s_{i,\varepsilon}(i=1,2)$ are bounded from above and below for $\varepsilon$ small enough.  Then we have
		\begin{equation}
			\frac{\partial F}{\partial t_1}(s_{1,\varepsilon},s_{2,\varepsilon})=	\frac{\partial F}{\partial t_2}(s_{1,\varepsilon},s_{2,\varepsilon})=0,
		\end{equation}
		which is equivalent to $$\sbr{s_{1,\varepsilon}u_{\theta_1},s_{2,\varepsilon}v_{\varepsilon}}\in \mathcal{N}.$$
		For $\varepsilon$ small enough, since $s_{i,\varepsilon}(i=1,2)$ are bounded from above and below,  then
		\begin{equation}\label{es-6}
			\theta_2 \log s_{2,\varepsilon}^2 \int_{\Omega}|v_\varepsilon|^2 =O(\varepsilon^2|\log \varepsilon|).
		\end{equation}
		By $\sbr{s_{1,\varepsilon}u_{\theta_1},s_{2,\varepsilon}v_{\varepsilon}}\in \mathcal{N}$, we have
		\begin{equation}
			\begin{aligned}
				s_{2,\varepsilon}^2&=\cfrac{\int_{\Omega}|\nabla v_\varepsilon|^2-\lambda_{2}\int_{\Omega}|v_\varepsilon|^2-\theta_2\int_{\Omega}\sbr{v_\varepsilon}^2\log \sbr{v_\varepsilon}^2-\theta_2\log s_{2,\varepsilon}^2 \int_{\Omega}|v_\varepsilon|^2-s_{1,\varepsilon}^2\beta \int_{\Omega}|u_{\theta_1}|^2|v_\varepsilon|^2}{\mu_2\int_{\Omega}|v_\varepsilon|^4}\\
				&=\cfrac{\mathcal{S}^2+O(\varepsilon^2|\log \varepsilon|)}{\mu_2\mathcal{S}^2+O(\varepsilon^4)} \to \frac{1}{\mu_2} \text{ as } \varepsilon \to 0^+,
			\end{aligned}
		\end{equation}
		and
		\begin{equation}
			\begin{aligned}
				0&=\int_{\Omega}|\nabla u_{\theta_1}|^2-\lambda_{1}\int_{\Omega}|u_{\theta_1}|^2-s_{1,\varepsilon}^2\mu_1 \int_{\Omega}|u_{\theta_1}|^4-\theta_1 \int_{\Omega}u_{\theta_1}^2\log u_{\theta_1}^2-\theta_1\log s_{1,\varepsilon}^2 \int_{\Omega}|u_{\theta_1}|^2-s_{2,\varepsilon}^2\beta \int_{\Omega}|u_{\theta_1}|^2|v_\varepsilon|^2\\
				&=\sbr{1-s_{1,\varepsilon}^2}\mu_1 \int_{\Omega}|u_{\theta_1}|^4-\theta_1\log s_{1,\varepsilon}^2 \int_{\Omega}|u_{\theta_1}|^2+O(\varepsilon^2|\log \varepsilon|).
			\end{aligned}
		\end{equation}
		Then we can see that $s_{1,\varepsilon}^2\to 1$ as $\varepsilon \to 0^+$. Therefore, for $\varepsilon$ small enough, we have
		\begin{equation} \label{f5}
			\frac{1}{2}\leq s_{1,\varepsilon}^2\leq 2 \ \ \text{ and } \ \ 	\frac{1}{2\mu_2}\leq s_{2,\varepsilon}^2\leq \frac{2}{\mu_2}.
		\end{equation}
		Since $\sbr{s_{1,\varepsilon}u_{\theta_1},s_{2,\varepsilon}v_{\varepsilon}}\in \mathcal{N}$, there holds
		\begin{equation}\label{f6}
			\mathcal{C}_{\mathcal{N}}\leq \mathcal{L}(s_{1,\varepsilon}u_{\theta_1},s_{2,\varepsilon}v_{\varepsilon})=: f_1(s_{1,\varepsilon})+f_2(s_{2,\varepsilon})-\frac{\beta}{2}s_{1,\varepsilon}^2s_{2,\varepsilon}^2\int_{\Omega}|u_{\theta_1}|^2|v_{\varepsilon}|^2,
		\end{equation}
		where
		\begin{equation}
			f_1(s_{1}):=\frac{1}{2}s_{1}^2\int_{\Omega}|\nabla u_{\theta_1}|^2-\frac{\lambda_1}{2}s_{1}^2\int_{\Omega}|u_{\theta_1}|^2-\frac{\mu_1}{4}s_{1}^4 \int_{\Omega}|u_{\theta_1}|^4 -\frac{\theta_1}{2}\int_{\Omega}\sbr{s_{1}u_{\theta_1}}^2(\log (s_{1}u_{\theta_1})^2-1),
		\end{equation}
		and
		\begin{equation}
			f_2(s_{2}):=\frac{1}{2}s_{2}^2\int_{\Omega} |\nabla v_{\varepsilon}|^2-\frac{\lambda_2}{2}s_{2}^2\int_{\Omega}|v_{\varepsilon}|^2-\frac{\mu_2}{4}s_{2}^4 \int_{\Omega}|v_{\varepsilon}|^4 -\frac{\theta_2}{2}\int_{\Omega}(s_{2}v_{\varepsilon})^2(\log (s_{2}v_{\varepsilon})^2-1).
		\end{equation}
		Recalling that $u_{\theta_1}$ is a positive least energy solution of $-\Delta u=\lambda_{1}u+\mu_1 |u|^2u+ \theta_1 u\log u^2$. Then, we have
		\begin{equation}\label{f3}
			\int_{\Omega}|\nabla u_{\theta_1}|^2=\lambda_{1}\int_{\Omega}|u_{\theta_1}|^2+\mu_1 \int_{\Omega}|u_{\theta_1}|^4+\theta_1 \int_{\Omega}u_{\theta_1}^2\log u_{\theta_1}^2,
		\end{equation}
		and
		\begin{equation}\label{f31}
			\begin{aligned}
				\mathcal{C}_{\theta_1}&=\frac{1}{2}\int_{\Omega}|\nabla u_{\theta_1}|^2-\frac{\lambda_1}{2}\int_{\Omega}|u_{\theta_1}|^2-\frac{\mu_1}{4}\int_{\Omega}|u_{\theta_1}|^4 -\frac{\theta_1}{2}\int_{\Omega}u_{\theta_1}^2(\log u_{\theta_1}^2-1).
			\end{aligned}
		\end{equation}
		By a direct calculation, we can see from \eqref{f3} that
		\begin{equation}
			\begin{aligned}
				f^\prime(s_{1})&=s_{1}\int_{\Omega}|\nabla u_{\theta_1}|^2-s_{1}\lambda_1\int_{\Omega}|u_{\theta_1}|^2-s_{1}^3\mu_1 \int_{\Omega}|u_{\theta_1}|^4 -\theta_1\int_{\Omega} s_{1}u_{\theta_1}\log (s_{1}u_{\theta_1})^2\\
				&=(s_{1}-s_{1}^3)\mu_1\int_{\Omega}|u_{\theta_1}|^4 -( s_{1}\log s_{1}^2)\theta_1 \int_{\Omega}|u_{\theta_1}|^2.
			\end{aligned}
		\end{equation}
		Then we can easily see that $f^\prime(s_{1})\geq 0$ for $0<s_1\leq 1$ and $f^\prime(s_{1})\leq 0$ for $s_1\geq 1$. Therefore, by \eqref{f31},
		\begin{equation}\label{C1}
			f_1(s_{1,\varepsilon})\leq f_1(1)=\mathcal{C}_{\theta_1}.
		\end{equation}
		On the other hand, by \eqref{es-5},\eqref{f5}, we have that
		\begin{equation}
			\begin{aligned}
				f_2(s_{2,\varepsilon})-\frac{\beta}{2}s_{1,\varepsilon}^2s_{2,\varepsilon}^2\int_{\Omega}|u_{\theta_1}|^2|v_{\varepsilon}|^2 & \leq f_2(s_{2,\varepsilon})+\frac{\theta_2}{2}s_{2,\varepsilon}^2\int_{\Omega}|v_\varepsilon|^2.
			\end{aligned}
		\end{equation}
		Therefore,
		\begin{equation}
			\begin{aligned}
				&	f_2(s_{2,\varepsilon})-\frac{\beta}{2}s_{1,\varepsilon}^2s_{2,\varepsilon}^2\int_{\Omega}|u_{\theta_1}|^2|v_{\varepsilon}|^2\\
				& \leq \frac{1}{2}s_{2,\varepsilon}^2\int_{\Omega} |\nabla v_{\varepsilon}|^2-\frac{\mu_2}{4}s_{2,\varepsilon}^4 \int_{\Omega}|v_{\varepsilon}|^4-\frac{\lambda_2-2\theta_2}{2}s_{2,\varepsilon}^2\int_{\Omega}|v_{\varepsilon}|^2 -\frac{\theta_2}{2}s_{2,\varepsilon}^2\log s_{2,\varepsilon}^2\int_{\Omega} |v_{\varepsilon}|^2-\frac{\theta_2}{2}s_{2,\varepsilon}^2\int_{\Omega} v_{\varepsilon}^2\log v_{\varepsilon}^2\\
				&\leq  \sbr{\frac{1}{2}s_{2,\varepsilon}^2-\frac{\mu_2}{4}s_{2,\varepsilon}^4}\mathcal{S}^2+O(\varepsilon^2)-\frac{s_{2,\varepsilon}^2}{2}\mbr{\sbr{\lambda_2-2\theta_2+\theta_2\log \frac{1}{2\mu_2}}\int_{\Omega}|v_{\varepsilon}|^2+\theta_2\int_{\Omega} (v_{\varepsilon})^2\log v_{\varepsilon}^2} \\
				&\leq \frac{1}{4}\mu_2^{-1}\mathcal{S}^2 -\frac{s_{2,\varepsilon}^2}{2}\mbr{\sbr{\lambda_2-2\theta_2+\theta_2\log \frac{1}{2\mu_2}} 8\omega_4\varepsilon^2|\log \varepsilon| + 8\theta_2 \log \sbr{\cfrac{8(\varepsilon^2+R^2)}{e(\varepsilon^2+4R^2)^2}}\omega_4 \varepsilon^2|\log \varepsilon|}+O(\varepsilon^2)\\
				&\leq \frac{1}{4}\mu_2^{-1}\mathcal{S}^2- 4s_{2,\varepsilon}^2 \theta_2\log \sbr{\cfrac{4 (\varepsilon^2+R^2)}{\mu_2e^{3-\lambda_{2}/\theta_2}(\varepsilon^2+4R^2)^{2}}}\omega_4 \varepsilon^2|\log \varepsilon|+O(\varepsilon^2)\\
				&\leq \frac{1}{4}\mu_2^{-1}\mathcal{S}^2- \frac{2\theta_2}{\mu_2}\log \sbr{\cfrac{4 }{25\mu_2e^{3-\lambda_{2}/\theta_2}R^2}}\omega_4 \varepsilon^2|\log \varepsilon|+O(\varepsilon^2)\\
				&<\frac{1}{4}\mu_2^{-1}\mathcal{S}^2,
			\end{aligned}
		\end{equation}
		where we choose $R>0$ small enough such that $\cfrac{4 }{25\mu_2e^{3-\lambda_{2}/\theta_2}R^2}>1$ and $\varepsilon<R$.
		
		Hence, by \eqref{f6} we have
		\begin{equation}
			\begin{aligned}
				\mathcal{C}_{\mathcal{N}} <\mathcal{C}_{\theta_1}+\frac{1}{4}\mu_2^{-1}\mathcal{S}^2.
			\end{aligned}	
		\end{equation}
		Similarly, we can also prove that $\mathcal{C}_{\mathcal{N}} <\mathcal{C}_{\theta_2}+\frac{1}{4}\mu_1^{-1}\mathcal{S}^2.$
		By  Proposition \ref{Lemma bound of single energy}, we can easily see that $\mathcal{C}_{\theta_i}<\frac{1}{4}\mu_i^{-1}\mathcal{S}^2$ for $i=1,2$. Therefore, we have
		\begin{equation}
			\mathcal{C}_{\mathcal{N}}< \min \lbr{ \mathcal{C}_{\theta_1}+\frac{1}{4}\mu_2^{-1}\mathcal{S}^2, \ \  \mathcal{C}_{\theta_2}+\frac{1}{4}\mu_1^{-1}\mathcal{S}^2,\ \ \frac{1}{4}\sbr{\mu_1^{-1}+\mu_2^{-1}}\mathcal{S}^2}.
		\end{equation}
		The proof is complete.
	\end{proof}
	Now we 	consider the following limit system
	\begin{equation} \label{System2}
		\begin{cases}
			-\Delta u=
			\mu_1|u|^{2}u+\beta |v|^{2}u, \ \ x\in \R^4\\
			-\Delta v=
			\mu_2|v|^{2}v+\beta |u|^{2}v, \ \ x\in \R^4\\
			u,v \in \mathcal{D}^{1,2}(\R^4).
		\end{cases}
	\end{equation}
	Define $\mathcal{D}=\mathcal{D}^{1,2}(\R^4)\times  \mathcal{D}^{1,2}(\R^4) $ and  a $C^2$-functional $\mathcal{E}: \mathcal{D} \to \R$ given by
	\begin{equation}
		\mathcal{E}(u,v)=\frac{1}{2} \int_{\R^4}(|\nabla u|^2+|\nabla v|^2) -\frac{1}{4}\int_{\R^4} \sbr{\mu_1|u|^4+2\beta |u|^2|v|^2 +\mu_2 |v|^4 }.
	\end{equation}
	We consider the level
	\begin{equation}\label{defi of A}
		\mathcal{A}=\inf_{(u,v)\in \widetilde{\mathcal{N}}} \mathcal{E}(u,v),
	\end{equation}
	with
	\begin{equation} \label{defi of widetilde{mathcal N}}
		\widetilde{\mathcal{N}}=\lbr{(u,v)\in \mathcal{D}:u\not\equiv0,v\not\equiv 0,\mathcal{E}^\prime(u,v)(u,0)=0, \mathcal{E}^\prime(u,v)(0,v)=0
		}.
	\end{equation}
	From  \cite{Zou 2012}, we know that
	\begin{equation} \label{f13}
		\mathcal{A}=\begin{cases}
			\frac{1}{4}(\mu_1^{-1}+\mu_2^{-2}) \mathcal{S}^2, &\quad \text{ if } \beta<0,\\
			\frac{1}{4}(k+l) \mathcal{S}^2, 	&\quad \text{ if } 0< \beta<\min\lbr{\mu_1,\mu_2}  \text{ or }\beta>\max\lbr{\mu_1,\mu_2} ,
		\end{cases}
	\end{equation}
	where $k,l>0$  satisfy
	\begin{equation} \label{f14}
		\begin{cases}
			&\mu_1k+\beta l=1,\\
			&\beta k+\mu_2 l=1.
		\end{cases}
	\end{equation}
	
	\begin{proposition} \label{Lemma energy estimate-2}
		Let $\beta \in   \sbr{-\beta_0,0}\cup \sbr{0,\beta_1}$. Then we have  $$\mathcal{C}_{\mathcal{N}}<\mathcal{A}.$$
	\end{proposition}
	
	\begin{proof}
		If $\beta \in  \sbr{-\beta_0,0}$,   we can see that the conclusion  follows directly from Proposition \ref{Lemma energy estimate-1}  and formula \eqref{f13}.  Now it remains to prove the case  $0<\beta<\beta_1$.
		
		Take arbitrary $y_0\in \Omega$, then there exists $R>0$ such that $B_{2R}(y_0)\subset \Omega$. Let $\xi\in C_0^\infty(\Omega)$ be the radial function, such that  $\xi(x)\equiv1$ for $0\leq |x-y_0|\leq R$, $0\leq \xi(x)\leq 1$ for $R\leq |x-y_0|\leq 2R$, $\xi(x)\equiv0$ for $ |x-y_0|\geq  2R$. Take
		\begin{equation}\label{f64}
			v_\varepsilon(x)=\xi(x) U_{\varepsilon,y_0}(x),
		\end{equation}
		where $ U_{\varepsilon,y_0}(x)$  is defined in \eqref{limit system_1}.
		Let
		\begin{equation} \label{f19}
			w_\varepsilon(x)=\sqrt{k}v_{\varepsilon}(x), \quad z_{\varepsilon}(x)=\sqrt{l}v_{\varepsilon}(x).
		\end{equation}
		Then by \eqref{es-1}, \eqref{es-2}, we have
		\begin{equation}\label{es-7}
			\begin{aligned}
				&\int_{\Omega} |\nabla w_{\varepsilon}|^2=k\mathcal{S}^2+O(\varepsilon^2),\quad 	\int_{\Omega} |\nabla z_{\varepsilon}|^2=l\mathcal{S}^2+O(\varepsilon^2),\\
				&\int_{\Omega} |w_\varepsilon|^4=k^2\mathcal{S}^2+O(\varepsilon^2), \quad \int_{\Omega} |z_\varepsilon|^4=l^2\mathcal{S}^2+O(\varepsilon^2),\\
				& \int_{\Omega} |w_\varepsilon|^2|z_\varepsilon|^2=kl\mathcal{S}^2+O(\varepsilon^2).
			\end{aligned}
		\end{equation}
		By a similar argument as used in that of Proposition \ref{Lemma energy estimate-1},  we can see that there exists $t_{1,\varepsilon},t_{2,\varepsilon}>0$, which are bounded from above and below for $\varepsilon$ small enough, such that $(t_{1,\varepsilon}	w_\varepsilon,t_{2,\varepsilon}z_{\varepsilon})\in \mathcal{N}$ and
		\begin{equation}
			\mathcal{L}(t_{1,\varepsilon}	w_\varepsilon,t_{2,\varepsilon}z_{\varepsilon})=\max_{t_1,t_2>0} \mathcal{L}(t_{1}	w_\varepsilon,t_{2}z_{\varepsilon}).
		\end{equation}
		Therefore, we have
		\begin{equation}\label{f18}
			\begin{aligned}
				\mathcal{C}_{\mathcal{N}}&\leq \mathcal{L}(t_{1,\varepsilon}	w_\varepsilon,t_{2,\varepsilon}z_{\varepsilon})\\
				&=\frac{1}{2}\int_{\Omega}\sbr{t_{1,\varepsilon}^2|\nabla w_\varepsilon|^2+t_{2,\varepsilon}^2 |\nabla z_{\varepsilon}|^2}-\frac{1}{4}\int_{\Omega}\sbr{t_{1,\varepsilon}^4\mu_1|w_\varepsilon|^4+2\beta t_{1,\varepsilon}^2t_{2,\varepsilon}^2|w_\varepsilon|^2|z_{\varepsilon}|^2+\mu_2t_{2,\varepsilon}^4|z_\varepsilon|^4}\\
				&-\frac{\lambda_1-\theta_1}{2}t_{1,\varepsilon}^2\int_{\Omega}|w_{\varepsilon}|^2 -\frac{\theta_1}{2}t_{1,\varepsilon}^2\log t_{1,\varepsilon}^2\int_{\Omega} |w_{\varepsilon}|^2-\frac{\theta_1}{2}t_{1,\varepsilon}^2\int_{\Omega} w_{\varepsilon}^2\log w_{\varepsilon}^2\\
				&-\frac{\lambda_2-\theta_2}{2}t_{2,\varepsilon}^2\int_{\Omega}|z_{\varepsilon}|^2 -\frac{\theta_2}{2}t_{2,\varepsilon}^2\log t_{2,\varepsilon}^2\int_{\Omega} |z_{\varepsilon}|^2-\frac{\theta_2}{2}t_{2,\varepsilon}^2\int_{\Omega} z_{\varepsilon}^2\log z_{\varepsilon}^2\\
				&=:I_1+I_2+I_3.
			\end{aligned}
		\end{equation}
		By \eqref{es-7}, we have
		\begin{equation}
			\begin{aligned}
				I_1&=\mbr{\frac{1}{2}\sbr{t_{1,\varepsilon}^2k+t_{2,\varepsilon}^2l}-\frac{1}{4}\sbr{\mu_1k^2t_{1,\varepsilon}^4+2\beta kl t_{1,\varepsilon}^2t_{2,\varepsilon}^2+\mu_2l^2t_{2,\varepsilon}^4}}(\mathcal{S}^2+O(\varepsilon^2)).
			\end{aligned}
		\end{equation}
		Consider
		\begin{equation}
			g(t_1,t_2)=\frac{1}{2}\sbr{t_{1}^2k+t_{2}^2l}-\frac{1}{4}\sbr{\mu_1k^2t_{1}^4+2\beta kl t_{1}^2t_{2}^2+\mu_2l^2t_{2}^4},
		\end{equation}
		then it is easy to see that there exist $\widetilde{t_1},\widetilde{t_2}>0$ such that
		\begin{equation}
			g(\widetilde{t_1},\widetilde{t_2})=\max_{t_1,t_2>0}g(t_1,t_2).
		\end{equation}
		Hence, combing \eqref{f14} with $	\frac{\partial }{\partial t_1}g(t_1,t_2)|_{(\widetilde{t_1},\widetilde{t_2})}=	\frac{\partial }{\partial t_2}g(t_1,t_2)|_{(\widetilde{t_1},\widetilde{t_2})}=0$, we can easily see that $(\widetilde{t_1},\widetilde{t_2})=(1,1)$.
		So, by \eqref{f13}, \eqref{f14} we have
		\begin{equation} \label{f15}
			\begin{aligned}
				I_1&\leq g(1,1)(\mathcal{S}^2+O(\varepsilon^2))\\
				&=\mbr{\frac{1}{2}\sbr{k+l}-\frac{1}{4}\sbr{\mu_1k^2+2\beta kl +\mu_2l^2}}(\mathcal{S}^2+O(\varepsilon^2))\\
				&=\frac{1}{4}\sbr{k+l}\mathcal{S}^2+O(\varepsilon^2)=\mathcal{A}+O(\varepsilon^2),		
			\end{aligned}
		\end{equation}
		On the other hand, since  $t_{i,\varepsilon}$ ($i=1,2$) are bounded from above and below for $\varepsilon$ small enough, we have
		\begin{equation}
			\log t_{i,\varepsilon}^2 \int_{\Omega}|v_\varepsilon|^2 =O(\varepsilon^2|\log \varepsilon|)\ \  \text{ for } \ i=1,2.
		\end{equation}
		Since $\sbr{t_{1,\varepsilon}w_\varepsilon,t_{2,\varepsilon}z_{\varepsilon}}\in \mathcal{N}$, it is  straightforward  to  show that
		\begin{equation}
			\begin{aligned}
				\mu_1 kt_{1,\varepsilon}^2+\beta l	t_{2,\varepsilon}^2&=\cfrac{\int_{\Omega}|\nabla v_\varepsilon|^2-\lambda_{1}\int_{\Omega}|v_\varepsilon|^2-\theta_1\int_{\Omega}v_\varepsilon^2\log v_\varepsilon^2-\theta_1 \log k \int_{\Omega} |v_\varepsilon|^2-\theta_1\log t_{1,\varepsilon}^2 \int_{\Omega}|v_\varepsilon|^2}{\int_{\Omega}|v_\varepsilon|^4}\\
				&=\cfrac{\mathcal{S}^2+O(\varepsilon^2|\log \varepsilon|)}{\mathcal{S}^2+O(\varepsilon^4)} \to 1 \text{ as } \varepsilon \to 0^+.
			\end{aligned}
		\end{equation}
		Similarly, we have
		\begin{equation}
			\beta k t_{1,\varepsilon}^2+\mu_2  l	t_{2,\varepsilon}^2 \to 1 \text{ as } \varepsilon \to 0^+.
		\end{equation}
		Combining these with \eqref{f14}, we can see that
		\begin{equation}
			t_{1,\varepsilon} \to 1, \quad t_{2,\varepsilon} \to 1 \ \ \text{ as } \varepsilon \to 0^+,
		\end{equation}
		which implies that
		\begin{equation}
			\log t_{i,\varepsilon}^2 \int_{\Omega}|v_\varepsilon|^2 =o(\varepsilon^2|\log \varepsilon|)\ \  \text{ for } \ i=1,2.
		\end{equation}
		Hence, by   \eqref{es-3}
		\begin{equation} \label{f16}
			\begin{aligned}
				I_2&=-\frac{\lambda_1-\theta_1}{2}t_{1,\varepsilon}^2\int_{\Omega}|w_{\varepsilon}|^2 -\frac{\theta_1}{2}t_{1,\varepsilon}^2\log t_{1,\varepsilon}^2\int_{\Omega} |w_{\varepsilon}|^2-\frac{\theta_1}{2}t_{1,\varepsilon}^2\int_{\Omega} w_{\varepsilon}^2\log w_{\varepsilon}^2\\
				&=-\frac{k(\lambda_1-\theta_1)+\theta_1 k\log  k}{2}t_{1,\varepsilon}^2\int_{\Omega}|v_{\varepsilon}|^2-\frac{\theta_1k}{2}t_{1,\varepsilon}^2\int_{\Omega} v_{\varepsilon}^2\log v_{\varepsilon}^2+o(\varepsilon^2|\log \varepsilon|)\\
				&\leq -\frac{t_{1,\varepsilon}^2}{2}\mbr{\sbr{k\lambda_1-\theta_1k+\theta_1 k\log  k} 8\omega_4\varepsilon^2|\log \varepsilon| + 8\theta_1k \log \sbr{\cfrac{8(\varepsilon^2+R^2)}{e(\varepsilon^2+4R^2)^2}}\omega_4 \varepsilon^2|\log \varepsilon|}+o(\varepsilon^2|\log \varepsilon|)\\
				&\leq -4t_{1,\varepsilon}^2\log\sbr{\cfrac{8k(\varepsilon^2+R^2)}{e^{2-\lambda_{1}/\theta_1}(\varepsilon^2+4R^2)^2}}\theta_1k\omega_4 \varepsilon^2|\log \varepsilon|+o(\varepsilon^2|\log \varepsilon|)\\
				&\leq -4t_{1,\varepsilon}^2\log\sbr{\cfrac{8k}{25e^{2-\lambda_{1}/\theta_1}R^2}}\theta_1k\omega_4 \varepsilon^2|\log \varepsilon|+o(\varepsilon^2|\log \varepsilon|)\\
				&\leq -C\omega_4 \varepsilon^2|\log \varepsilon|+o(\varepsilon^2|\log \varepsilon|).
			\end{aligned}
		\end{equation}
		Similarly, we have
		\begin{equation}\label{f17}
			\begin{aligned}
				I_3&\leq -4t_{2,\varepsilon}^2\log\sbr{\cfrac{8l}{25e^{2-\lambda_{2}/\theta_2}R^2}}\theta_2l\omega_4 \varepsilon^2|\log \varepsilon|+o(\varepsilon^2|\log \varepsilon|)\\
				&\leq -C\omega_4 \varepsilon^2|\log \varepsilon|+o(\varepsilon^2|\log \varepsilon|),
			\end{aligned}
		\end{equation}
		where $C>0$ and we choose $R>0$ small enough such that $\cfrac{8k}{25e^{2-\lambda_{1}/\theta_1}R^2}>1$,  $\cfrac{8l}{25e^{2-\lambda_{2}/\theta_2}R^2}>1$ and take  $\varepsilon<R$.
		Therefore, by\eqref{f18}, \eqref{f15}, \eqref{f16}, \eqref{f17},  we can see that
		\begin{equation}
			\begin{aligned}
				\mathcal{C}_{\mathcal{N}}\leq \mathcal{A}-C\omega_4 \varepsilon^2|\log \varepsilon|+o(\varepsilon^2|\log \varepsilon|)<\mathcal{A} \quad \text{ for } \varepsilon >0 \text{ small enough}.
			\end{aligned}
		\end{equation}
		The proof is completed.
	\end{proof}

	%

	%

	\section{Proof of Theorem \ref{Theorem-1}}\label{Sect3}
	
	\begin{proof}[\bf Proof of (1) and (2) in  Theorem \ref{Theorem-1}]
		
		Repeating the Scheme of the proof of Theorem 1.3 in \cite{Zou 2012} with some slight modifications, we can construct a Palais-Smale sequence at the  level $\mathcal{C}_{\mathcal{N}}$. Then
		there exists a sequence $\lbr{(u_n,v_n)}\subset \mathcal{N}$  satisfying
		\begin{equation}\label{f7}
			\lim_{n\to \infty}\mathcal{L}(u_n,v_n)=\mathcal{C}_{\mathcal{N}},\quad  \lim_{n\to \infty}\mathcal{L}^\prime(u_n,v_n)=0.
		\end{equation}
		By Proposition \ref{Lemma energy estimate-1}, we can see that $\mathcal{L}(u_n,v_n)\leq \frac{1}{2}\sbr{\mu_1^{-1}+\mu_2^{-1}}\mathcal{S}^2$ for $n$ large enough. Then by Lemma \ref{Lemma L^4 norm bounded} and \eqref{f8}, we can see that $\lbr{(u_n,v_n)}$ is bounded in $\mathcal{H} $. Hence, we may assume that
		\begin{equation}
			\sbr{u_n,v_n} \rightharpoonup (u,v) \text{ weakly in } \mathcal{H}.
		\end{equation}
		Passing to subsequence, we may also  assume that
		\begin{equation}
			\begin{aligned}
				&u_n \rightharpoonup u, \quad v_n\rightharpoonup v \ \text{ weakly in } L^4(\Omega),\\
				&u_n \to u, \quad v_n\to v \ \text{ strongly in } L^p(\Omega) \text{ for } 2\leq p<4,\\
				&u_n \to u, \quad v_n \to v \ \text{ almost everywhere in } \Omega.
			\end{aligned}
		\end{equation}
		By  using the inequality $|s^2\log s^2| \leq Cs^{2-\tau}+Cs^{2+\tau}$, $\tau\in \sbr{0,1}$ and the dominated convergence theorem, one gets
		\begin{equation}
			\lim_{n\to \infty} \int_{\Omega} u_n^+\varphi^+\log(u_n^+)^2 =\int_{\Omega} u^+\varphi^+\log (u^+)^2 \ \  \text{ for any } \varphi\in C_0^\infty(\Omega) .
		\end{equation}
		Then by \eqref{f7}, we have $\mathcal{L}^\prime (u,v)=0$. Moreover, by using the weak-lower semicontinuity  of the norm, we have
		\begin{equation}\label{f53}
			\mathcal{L}(u,v)\leq \frac{1}{2}\sbr{\mu_1^{-1}+\mu_2^{-1}}\mathcal{S}^2.
		\end{equation}
		Let $w_n=u_n-u$ and $z_n=v_n-v$. Using the Br\'{e}zis-Lieb Lemma (see \cite{Brezis Lieb lemma} and  \cite[formula 5.37]{Zou 2012}). we have
		\begin{equation}\label{f9}
			\begin{aligned}
				&|u_n^+|^4_4=|u^+|_4^4+|w_n^+|_4^4+o_n(1), \quad |v_n^+|^4_4=|v^+|_4^4+|z_n^+|_4^4+o_n(1),\\
				&|u_n^+v_n^+|_2^2=|u^+v^+|_2^2+|w_n^+z_n^+|_2^2+o_n(1).
			\end{aligned}
		\end{equation}
		Since $(u_n,v_n)\in \mathcal{N}$ and $\mathcal{L}^\prime(u,v)=0$, 	by \eqref{f9} and  \cite[Lemma 2.3]{Deng-He-Pan=Arxiv=2022}, we have
		\begin{equation} \label{f11}
			\begin{aligned}
				|\nabla w_n|_2^2=\mu_1|w_n^+|_4^4+\beta|w_n^+z_n^+|_2^2+o_n(1), \quad 	|\nabla z_n|_2^2=\mu_2|z_n^+|_4^4+\beta|w_n^+z_n^+|_2^2+o_n(1).
			\end{aligned}
		\end{equation}
		By a direct calculation, one gets
		\begin{equation}\label{f10}
			\mathcal{L}(u_n,v_n)=\mathcal{L}(u,v)+\frac{1}{4}\int_{\Omega}|\nabla w_n|^2+\frac{1}{4}\int_{\Omega}|\nabla z_n|^2+o_n(1).
		\end{equation}
		Passing to a subsequence, we may assume that
		\begin{equation}
			\int_{\Omega}|\nabla w_n|^2=k_1+o_n(1), \quad \int_{\Omega}|\nabla z_n|^2=k_2+o_n(1).
		\end{equation}	
		Letting $n\to +\infty$ in \eqref{f10}, we have
		\begin{equation}\label{f12}
			0 \leq \mathcal{L}(u,v) \leq \mathcal{L}(u,v)+\frac{1}{4}k_1+\frac{1}{4}k_2=\lim_{n\to \infty} 	\mathcal{L}(u_n,v_n) =\mathcal{C}_{\mathcal{N}}.
		\end{equation}
		Now we claim that $u\not \equiv 0$ and $v\not \equiv 0$.	
		\vskip 0.1in
		{\bf Case 1.} $u\equiv0$ and $v\equiv 0$.
		
		Firstly, we prove that $k_1>0$ and $k_2>0$.  Without loss of generality, we assume by contradiction that $k_1=0$, then we can see that $w_n\to 0$ strongly in $H_0^1(\Omega)$ and $u_n \to 0$ strongly in $H_0^1(\Omega)$, then by the Sobolev inequality, we can see that $u_n \to 0$ strongly in $L^4(\Omega)$, which is impossible by Lemma \ref{Lemma L^4 norm bounded}. Therefore,  we have that $k_1>0$ and $k_2>0$.  Since \eqref{f11} holds, it is easy to check that there exists $t_n,s_n>0$ such that $(t_nw_n,s_nz_n) \in \widetilde{\mathcal{N}}$, which is given by \eqref{defi of widetilde{mathcal N}}. Moreover,
		$	t_n=1+o_n(1),   s_n=1+o_n(1).$
		Therefore, by \eqref{f13}  we have
		\begin{equation}
			\frac{1}{4}k_1+\frac{1}{4}k_2=\lim_{n\to \infty} \mathcal{E}(w_n,z_n)=\lim_{n\to \infty} \mathcal{E}(t_nw_n,s_nz_n)\geq \mathcal{A}.
		\end{equation}
		Hence, by \eqref{f12}, we know that $\mathcal{C}_{\mathcal{N}}\geq\mathcal{A}$, a contradiction with Proposition \ref{Lemma energy estimate-2}. Therefore, Case 1 is impossible.
		
		\vskip 0.1in

		{\bf Case 2.} $u\equiv0$, $v\not \equiv 0$ or $u \not \equiv0$, $v\equiv 0$.
		
		Without loss of generality, we may assume that $u\equiv0$, $v\not \equiv 0$. Then by Case 1, we have that $k_1>0$, and we may assume that $k_2=0$. Then we know that $ |w_n^+z_n^+|_2^2=o_n(1)$. By \eqref{f11}, we have
		\begin{equation}
			\int_{\Omega}|\nabla w_n|^2 =\mu_1|w_n^+|_4^4+o_n(1)\leq \mu_1\mathcal{S}^{-2}\sbr{\int_{\Omega}|\nabla w_n|^2 }^2.
		\end{equation}
		Thus, letting $n\to \infty$, we have  $k_1\geq \mu_1^{-1}\mathcal{S}^{2}$.  Notice that $v$ is a solution of $-\Delta w =\lambda_{2}w +\mu_2|w|^2w+\theta_2w\log w^2$, we have $\mathcal{L}(0,v)\geq \mathcal{C}_{\theta_2}$. Therefore, by \eqref{f12} we have that
		\begin{equation}
			\mathcal{C}_{\mathcal{N}}\geq \mathcal{C}_{\theta_2} +\frac{1}{4}  \mu_1^{-1}\mathcal{S}^{2},
		\end{equation}
		which is a contradiction with Proposition \ref{Lemma energy estimate-1}. Therefore, Case 2 is impossible.
		
		Since Case 1 and 2 are both impossible, we get that $u\not \equiv 0$  and $v\not \equiv 0$. Therefore, $ (u,v)\in \mathcal{N}$ and by \eqref{f12} we have that $\mathcal{L}(u,v)=\mathcal{C}_{\mathcal{N}}$.
		Then combining \eqref{f53}  with Proposition \ref{Lemma natural constraint}, \ref{Lemma natural constraint2} and  Lemma \ref{Lemma natural constraint3}, $(u,v)$ is a solution of system \eqref{System1}. Since  $\mathcal{L}^\prime(u,v)=0$, we can see that
		\begin{equation}
			0=\mathcal{L}^\prime(u,v)(u^-,0)=\int_{\Omega}|\nabla u^-|^2, \quad 	0=\mathcal{L}^\prime(u,v)(0,v^-)=\int_{\Omega}|\nabla v^-|^2.
		\end{equation}
		which implies that $u\geq 0$, $v\geq 0$. By the Morse's iteration, the solutions $u,v$ belong to $ L^{\infty}(\Omega)$. Then the H\"{o}lder estimate implies that $u,v\in C^{0,\gamma}(\Omega)$ for any $0<\gamma<1$. Define  $g_i:[0,+\infty) \to \R$, $i=1,2$  by
		\begin{equation}
			g_i(s):=\begin{cases}
				2\theta_i|s\log s^2|, \quad& s>0,\\
				0, \quad &s=0.
			\end{cases}
		\end{equation}
		Then we follow the arguments  in  \cite{Deng-He-Pan=Arxiv=2022,vazquez=AMO=1984}, and get that $u ,v\in C^2(\Omega)$ and $u,v>0$ in $\Omega$. This completes the proof.
	\end{proof}

	It remains to prove (3) of Theorem \ref{Theorem-1}.  We will follow the strategies in \cite{Zou 2012} and  postpone the proof to the end of this section. Before proceeding, we introduce some definitions and lemma to the proof. From now, we assume that $\beta>\max\lbr{\mu_1,\mu_2}$.
	Let
	\begin{equation}
		\mathcal{C}_M:=\inf_{\gamma\in\Gamma}\max_{t\in \mbr{0,1}} \mathcal{L}(\gamma(t)),
	\end{equation}
	where $\Gamma=\lbr{\gamma \in C([0,1],\mathcal{H}): \gamma(0)=0, \mathcal{L}(\gamma(1))<0}$.  By a similar argument as the one used in Proposition \ref{Lemma energy estimate-1}, we can see that  for any $(u,v)\in \mathcal{H}$ with $(u,v)\neq (0,0)$, there exists $s_{u,v}>0$ such that
	\begin{equation}
		\begin{aligned}
			\max_{t>0} \mathcal{L}(tu,tv) &= \mathcal{L}(s_{u,v}u,s_{u,v}v).
		\end{aligned}
	\end{equation}
	Moreover, we have $(s_{u,v}u,s_{u,v}v) \in \mathcal{M}$, where
	\begin{equation}
		\mathcal{M}=\lbr{(u,v)\in \mathcal{H}\setminus \lbr{(0,0)}: \mathcal{L}^{\prime}(u,v)(u,v)=0}.
	\end{equation}
	Notice that $\theta_{i}>0$ for $i=1,2$, it is not difficult to check that
	\begin{equation} \label{f41}
		\mathcal{C}_M=\inf_{(u,v)\in \mathcal{H}\setminus \lbr{(0,0)}} \max_{t>0} \ \mathcal{L}(tu,tv)=\inf_{(u,v)\in \mathcal{M}}\mathcal{L}(u,v).
	\end{equation}
	Since $\mathcal{N}\subset \mathcal{M}$, one has that
	\begin{equation}\label{f29}
		\mathcal{C}_M\leq \mathcal{C}_{\mathcal{N}}.
	\end{equation}
	Recall \eqref{f18} -- \eqref{f17}, we can also prove that if $\beta>\max \lbr{\mu_1,\mu_2}$,
	\begin{equation} \label{f21}
		\mathcal{C}_M\leq \max_{t >0}\mathcal{L}(tw_\varepsilon,tv_\varepsilon)\leq \max_{t_1,t_2>0}\mathcal{L}(t_1w_\varepsilon,t_2v_\varepsilon)<\mathcal{A},
	\end{equation}
	where $w_\varepsilon,v_\varepsilon$  is defined in \eqref{f19}.
	
	\begin{lemma} \label{Lemma mountain pass structure}
		The functional $\mathcal{L}$ has a mountain pass geometry structure, that is,
		
		\begin{enumerate}
			\item[(i) ]  there exists $\alpha, \zeta>0$, such that $\mathcal{L}(u,v)\geq \alpha>0 $ for all $	\left\| (u,v)\right\|_{\mathcal{H}} =\zeta$;
			\item[(ii) ] there exists $(w,z) \in \mathcal{H}$, such that $\left\| (w,z)\right\|_{\mathcal{H}} \geq \zeta $  and $\mathcal{L}(w,z)<0$,
		\end{enumerate}
		where  $	\left\| (u,v)\right\|^2_{\mathcal{H}} := \int_{\Omega} \sbr{|\nabla u|^2+|\nabla v|^2}$.

	\end{lemma}
	
	\begin{proof}
		Since  $\theta_i>0$, it follows from the inequality $s^2\log s^2 \leq e^{-1}s^4$ for all $s>0$ that
		\begin{equation}
			\begin{aligned}
				&\frac{\lambda_1}{2}\int_{\Omega}|u^+|^2+\frac{\theta_1}{2}\int_{\Omega}(u^+)^2(\log (u^+)^2-1)\\&=\frac{\theta_1}{2} \int_{\Omega}(u^+)^2(\log (u^+)^2+\frac{\lambda_{1}}{\theta_1}-1) \leq \frac{\theta_1}{2}\int_{\Omega}(u^+)^2(\log e^{\frac{\lambda_{1}}{\theta_1}}(u^+)^2)\\
				&\leq  \frac{\theta_1}{2} e^{\frac{\lambda_{1}}{\theta_1}-1} \int_{\Omega}|u^+|^4 \leq \frac{\theta_1}{2} e^{\frac{\lambda_{1}}{\theta_1}-1} \mathcal{S}^{-2} \sbr{\int_{\Omega} |\nabla u|^2}^2.
			\end{aligned}
		\end{equation}
		Similarly, we have
		\begin{equation}
			\frac{\lambda_2}{2}\int_{\Omega}|v^+|^2+\frac{\theta_2}{2}\int_{\Omega}(v^+)^2(\log (v^+)^2-1) \leq \frac{\theta_2}{2} e^{\frac{\lambda_{2}}{\theta_2}-1} \mathcal{S}^{-2} \sbr{\int_{\Omega} |\nabla v|^2}^2.
		\end{equation}
		Notice that $\frac{\beta}{2}\int_{\Omega}|u^+|^2|v^+|^2  \leq \frac{\beta}{4}\sbr{\int_{\Omega}|u^+|^4+|v^+|^4}$ and $ \int_{\Omega}|u^+|^4 \leq \mathcal{S}^{-2} \sbr{\int_{\Omega} |\nabla u|^2}^2$,  we have
		\begin{equation} \label{}
			\begin{aligned}
				\mathcal{L}(u,v)
				&\geq \frac{1}{2}\int_{\Omega} |\nabla u|^2-\sbr{\frac{\mu_1+\beta}{4} +\frac{\theta_1}{2} e^{\frac{\lambda_{1}}{\theta_1}-1} }\mathcal{S}^{-2} \sbr{\int_{\Omega} |\nabla u|^2}^2\\
				&+\frac{1}{2}\int_{\Omega} |\nabla v|^2-\sbr{\frac{\mu_2+\beta}{4} +\frac{\theta_2}{2} e^{\frac{\lambda_{2}}{\theta_2}-1}} \mathcal{S}^{-2} \sbr{\int_{\Omega} |\nabla v|^2}^2\\
				&\geq  \frac{1}{2} 	\left\| (u,v)\right\|_{\mathcal{H}}-C	\left\| (u,v)\right\|_{\mathcal{H}}^2,
			\end{aligned}
		\end{equation}
		which implies that there exists $\alpha>0$ and $\zeta>0$ such that $\mathcal{L}(u,v)\geq \alpha>0 $ for all $	\left\| (u,v)\right\|_{\mathcal{H}} =\zeta$.
		
		On the other hand, let $\varphi_1,\varphi_2 \in H_0^1(\Omega)\setminus \lbr{0}$ be fixed positive functions, then  for any $t>0$, we have
		\begin{equation}
			\begin{aligned}
				\mathcal{L}(t\varphi_1,t\varphi_2)&=\frac{t^2}{2}\int_{\Omega}\sbr{|\nabla \varphi_1|^2+|\nabla \varphi_2|^2}-\frac{t^4}{2} \int_{\Omega}\sbr{\mu_1|\varphi_1|^4+2\beta|\varphi_1|^2|\varphi_2|^2+\mu_2 |\varphi_2|^4}\\
				&-\frac{t^2\log t^2}{2} \int_{\Omega}\sbr{\theta_1\varphi_1^2+\theta_2\varphi_2^2}-\frac{t^2}{2}\int_{\Omega}\sbr{\theta_1\varphi_1^2\log \sbr{e^{\frac{\lambda_{1}}{\theta_1}-1}\varphi_1^2}+\theta_2\varphi_2^2\log \sbr{e^{\frac{\lambda_{2}}{\theta_2}-1}\varphi_2^2}}\\
				&\to -\infty \quad \text{ as } t\to +\infty,
			\end{aligned}
		\end{equation}
		since  $\lim_{t \to +\infty}\frac{t^4}{t^2\log t^2} =+\infty$. Therefore, we can choose $t_0>0$ large enough such that
		\begin{equation}
			\mathcal{L}(t_0\varphi_1,t_0\varphi_2)<0, \quad \text{ and } \quad 	\left\| (t_0\varphi_1,t_0\varphi_2)\right\|_{\mathcal{H}} > \zeta.
		\end{equation}
		This completes the proof.
	\end{proof}

	\begin{proposition} \label{Lemma energy estimate-3}
		Let $\beta_2$ be the largest root of the equation
		\begin{equation}
			h(\beta):=\beta^2-\frac{2}{\Lambda}\beta +\frac{\mu_1+\mu_2}{\Lambda}-\mu_1\mu_2=0,
		\end{equation}
		where $\Lambda$ is defined in \eqref{defi of Lambda}. Then $\beta_2>\max\lbr{\mu_1,\mu_2}$, and for any $\beta>\beta_2$, we have
		\begin{equation}
			\mathcal{C}_M< \min \lbr{\mathcal{C}_{\theta_1},\mathcal{C}_{\theta_2}}.
		\end{equation}
	\end{proposition}
	
	\begin{proof}
		Without loss of generality, we assume that $ \mu_2:=\max\lbr{\mu_1,\mu_2}$, then by \eqref{defi of Lambda}, we can see that $\Lambda < \mu_2^{-1}$,  and  so $h(\mu_2)\leq 0$, which implies that $\beta\geq \mu_2=\max\lbr{\mu_1,\mu_2}$.
		
		For any $\beta>\beta_2$, we have  $h(\beta)>0$ and
		\begin{equation}\label{f20}
			\Lambda >\cfrac{2\beta-\mu_1-\mu_2}{\beta^2-\mu_1\mu_2}=k+l,
		\end{equation}
		where $k,l$ is given by \eqref{f14}. Hence, by  \eqref{f13}, \eqref{f21}, \eqref{f20} and Proposition \ref{Lemma bound of single energy}, we have
		\begin{equation}\label{f22}
			\min \lbr{\mathcal{C}_{\theta_1},\mathcal{C}_{\theta_2}} \geq \frac{1}{4} \Lambda \mathcal{S}^2 >\frac{1}{4} (k+l)\mathcal{S}^2 =\mathcal{A}>\mathcal{C}_M.
		\end{equation}
		Therefore, the proof is complete.
	\end{proof}
	
	\begin{lemma} \label{Lemma L^4 bounded in M}
		Assume that $\beta>\beta_2$, then there exist  a constant  $C_3>0$, such that for any $(u,v)\in \mathcal{M}$,  there holds
		\begin{equation}
			\int_{\Omega} (|u^+|^4+|v^+|^4)\geq C_3.
		\end{equation}
		Here, $C_3$   depends on $\beta$, $\mu_i$,  $\lambda_{i}$, $i=1,2$.
	\end{lemma}
	\begin{proof}
		%
		Since $(u,v)\in \mathcal{M}$,  we have
		\begin{equation}
			\begin{aligned}
				\mathcal{S}(|u^+|_4^2+|v^+|_4^2)& \leq |\nabla u|_2^2+|\nabla v|_2^2\\&=\mu_1|u^+|_4^4+2\beta|u^+v^+|_2^2+\mu_2|v^+|_4^4 +\theta_1\int_{\Omega}( u^+)^2\log (e^{\frac{\lambda_{1}}{\theta_1}}(u^+)^2)+\theta_2\int_{\Omega}( v^+)^2\log (e^{\frac{\lambda_{2}}{\theta_2}}(v^+)^2)\\
				&\leq (\mu_1+\beta+\theta_1e^{\frac{\lambda_{1}}{\theta_1}-1} )|u^+|^4+(\mu_2+\beta+\theta_2e^{\frac{\lambda_{2}}{\theta_2}-1})|v^+|_4^4\\
				&\leq C(|u^+|_4^4+|v^+|_4^4).
			\end{aligned}
		\end{equation}
		Therefore, there exists a constant $C_3>0$ such that $\int_{\Omega} (|u^+|^4+|v^+|^4) \geq C_3.$ This completes the proof.
	\end{proof}

	%
	%

	\begin{proof}[\bf Proof of (3) in  Theorem \ref{Theorem-1}]
		Let $\beta>\beta_2$. By Lemma \ref{Lemma mountain pass structure} and the mountain pass theorem (see  \cite{Ambrosotti=JFA=1973,william=1996}), there exists a sequence $\lbr{(u_n,v_n)}\subset \mathcal{H}$ such that
		\begin{equation} \label{f35}
			\lim_{n\to \infty} \mathcal{L}(u_n,v_n)=\mathcal{C}_M, \quad \lim_{n\to \infty} \mathcal{L}^\prime(u_n,v_n)=0.
		\end{equation}
		By \eqref{f22}, we know that $\mathcal{L}(u_n,v_n)\leq  \frac{1}{2}(k+l)\mathcal{S}^2$ for $n$ large enough. Then we can see  from \eqref{f23} and  \eqref{f8} that $\lbr{(u_n,v_n)}$ is bounded in $\mathcal{H}$. So, we may assume that 	\begin{equation}
			\sbr{u_n,v_n} \rightharpoonup (u,v) \text{ weakly in } \mathcal{H}.
		\end{equation}
		Let $w_n=u_n-u$ and $z_n=v_n-v$.
		For simplicity,  we still use the same symbol as  in the proof of (1) and (2) in Theorem \ref{Theorem-1}. So,
		we have $\mathcal{L}^\prime (u,v)=0$ and
		\begin{equation}\label{f28}
			0 \leq \mathcal{L}(u,v) \leq \mathcal{L}(u,v)+\frac{1}{4}(k_1+k_2)=\lim_{n\to \infty} 	\mathcal{L}(u_n,v_n) =\mathcal{C}_M<\mathcal{A}.
		\end{equation}
		Next,  we will show that $u\not \equiv 0$ and $v\not \equiv 0$.	
		
		{\bf Case 1.} $u\equiv0$ and $v\equiv 0$.
		
		Firstly, we deduce from \eqref{f1}, \eqref{f41} and  Lemma \ref{Lemma L^4 bounded in M} that   $\mathcal{C}_M>0$. Then by \eqref{f28}, we have   $k_1+k_2=4\mathcal{C}_M>0$.  Then by a similar   argument as used in the proof of \cite[Theorem 1.3]{Zou 2012}, we can see that
		\begin{equation}
			\mathcal{C}_M\geq \mathcal{A},
		\end{equation}
		a contradiction with \eqref{f21}. Therefore, Case 1 is impossible.
		
		{\bf Case 2.} $u\equiv0$, $v\not \equiv 0$ or $u \not \equiv0$, $v\equiv 0$.
		
		Without loss of generality, we may assume that $u\equiv0$, $v\not \equiv 0$.  Notice that $v$ is a solution of $-\Delta w =\lambda_{2}w +\mu_2|w|^2w+\theta_2w\log w^2$, we have $\mathcal{L}(0,v)\geq \mathcal{C}_{\theta_2}$. Therefore, by \eqref{f28} we have that 	$\mathcal{C}_M\geq \mathcal{L}(0,v)\geq  \mathcal{C}_{\theta_2}  $, which is a contradiction with Proposition \ref{Lemma energy estimate-3}. Therefore, Case 2 is impossible.
		
		Since Case 1 and 2 are both impossible, we get that $u\not \equiv 0$  and $v\not \equiv 0$. Since  $\mathcal{L}^\prime(u,v)=0$, we can see that $\sbr{u,v}\in \mathcal{N}$. Combining with this fact,  we deduce from \eqref{f29} and \eqref{f28} that $\mathcal{C}_M\leq \mathcal{C}_{\mathcal{N}}\leq \mathcal{L}(u,v)\leq \mathcal{C}_M$.  Hence, $\mathcal{L}(u,v)=\mathcal{C}_M=\mathcal{C}_{\mathcal{N}}$. By $\mathcal{L}^\prime (u,v)=0$ , we know that $(u,v)$ is a  least energy solution of \eqref{System1}. By a similar argument as used in the proof of (1)-(2) in Theorem \ref{Theorem-1}, we can show that $u,v>0$ and $u,v \in C^2(\Omega)$.  This completes the proof.
	\end{proof}

	\section{Proof of Theorem \ref{Theorem-2}}  \label{Sect4}
	
	\begin{lemma} \label{lm}
		When $\beta < 0$, we assume that one of the following holds:
		\begin{itemize}
			\item[(i)] $(\lambda_1,\mu_1,\theta_1; \lambda_2,\mu_2,\theta_2) \in A_1$,
			\item[(ii)] $(\lambda_1,\mu_1,\theta_1; \lambda_2,\mu_2,\theta_2) \in A_2$,
			\item[(iii)] $(\lambda_1,\mu_1,\theta_1; \lambda_2,\mu_2,\theta_2) \in A_3$;
		\end{itemize}
	    when $\beta > 0$, we assume that there exists $\epsilon > 0$ such that one of the following holds:
	    \begin{itemize}
	    	\item[(iv)] $(\lambda_1,\mu_1+\beta\epsilon,\theta_1; \lambda_2,\mu_2+\frac{\beta}{\epsilon},\theta_2) \in A_1$,
	    	\item[(v)] $(\lambda_1,\mu_1+\beta\epsilon,\theta_1; \lambda_2,\mu_2+\frac{\beta}{\epsilon},\theta_2) \in A_2$,
	    	\item[(vi)] $(\lambda_1,\mu_1+\beta\epsilon,\theta_1; \lambda_2,\mu_2+\frac{\beta}{\epsilon},\theta_2) \in A_3$.
	    \end{itemize}
    Then there exist $\delta, \rho > 0$ such that $\mathcal{L}(u,v) \geq \delta$ for all $\sqrt{|\nabla u|_2^2+|\nabla v|_2^2} = \rho$.
	\end{lemma}

    \begin{proof}
    	{\bf Case 1.} $\beta < 0$.
    	
        In this case, we have
        \begin{align}
        	\mathcal{L}(u,v)&\geq\frac{1}{2}\int_{\Omega} |\nabla u|^2-\frac{\lambda_1}{2}\int_{\Omega}|u^+|^2-\frac{\mu_1}{4} \int_{\Omega}|u^+|^4-\frac{\theta_1}{2}\int_{\Omega}(u^+)^2(\log (u^+)^2-1)\\&+\frac{1}{2}\int_{\Omega} |\nabla v|^2-\frac{\lambda_2}{2}\int_{\Omega}|v^+|^2-\frac{\mu_2}{4} \int_{\Omega}|v^+|^4-\frac{\theta_2}{2}\int_{\Omega}(v^+)^2(\log (v^+)^2-1).
        \end{align}
        Like the proof of \cite[Lemma 2.1]{Deng-He-Pan=Arxiv=2022}, one can obtain:

    	(i) If $(\lambda_1,\mu_1,\theta_1; \lambda_2,\mu_2,\theta_2) \in A_1$, then
    	\begin{align}
    		&\mathcal{L}(u,v)\\
    		\geq& \frac12\frac{\lambda_1(\Omega)-\lambda_1}{\lambda_1(\Omega)}|\nabla u|_2^2 - \frac{\mu_1}{4S^2}|\nabla u|_2^4 + \frac{\theta_1}{2}|\Omega| + \frac12\frac{\lambda_1(\Omega)-\lambda_2}{\lambda_1(\Omega)}|\nabla v|_2^2 - \frac{\mu_2}{4S^2}|\nabla v|_2^4 + \frac{\theta_2}{2}|\Omega| \\
    		\geq& \frac12\min\{\frac{\lambda_1(\Omega)-\lambda_1}{\lambda_1(\Omega)},\frac{\lambda_1(\Omega)-\lambda_2}{\lambda_1(\Omega)}\}(|\nabla u|_2^2+|\nabla v|_2^2) - \frac{\max\{\mu_1,\mu_2\}}{4S^2}(|\nabla u|_2^2+|\nabla v|_2^2)^2 + \frac{\theta_1+\theta_2}{2}|\Omega|.
    	\end{align}
    	Let $\rho = \sqrt{\frac{\min\{\lambda_1(\Omega)-\lambda_1,\lambda_1(\Omega)-\lambda_2\}}{\lambda_1(\Omega)\max\{\mu_1,\mu_2\}}}S$ and $\delta = \frac14 \frac{(\min\{\lambda_1(\Omega)-\lambda_1,\lambda_1(\Omega)-\lambda_2\})^2}{\lambda_1(\Omega)^2\max\{\mu_1,\mu_2\}}S^2 + \frac{\theta_1+\theta_2}{2}|\Omega|$, we deduce that
    	\begin{align}
    		\mathcal{L}(u,v) \geq \frac14 \frac{(\min\{\lambda_1(\Omega)-\lambda_1,\lambda_1(\Omega)-\lambda_2\})^2}{\lambda_1(\Omega)^2\max\{\mu_1,\mu_2\}}S^2 + \frac{\theta_1+\theta_2}{2}|\Omega|= \delta > 0.
    	\end{align}
    	
    	(ii) If $(\lambda_1,\mu_1,\theta_1; \lambda_2,\mu_2,\theta_2) \in A_2$, then
    	\begin{align}
    		&\mathcal{L}(u,v)\\
    		\geq& \frac12\frac{\lambda_1(\Omega)-\lambda_1}{\lambda_1(\Omega)}|\nabla u|_2^2 - \frac{\mu_1}{4S^2}|\nabla u|_2^4 + \frac{\theta_1}{2}|\Omega| + \frac12|\nabla v|_2^2 - \frac{\mu_2}{4S^2}|\nabla v|_2^4 + \frac{\theta_2}{2}e^{-\frac{\lambda_{2}}{\theta_{2}}}|\Omega| \\
    		\geq& \frac12\frac{\lambda_1(\Omega)-\lambda_1}{\lambda_1(\Omega)}(|\nabla u|_2^2+|\nabla v|_2^2) - \frac{1}{4S^2}\max\{\mu_1,\mu_2\}(|\nabla u|_2^2+|\nabla v|_2^2)^2 + \frac{\theta_1+\theta_2e^{-\frac{\lambda_{2}}{\theta_{2}}}}{2}|\Omega|.
    	\end{align}
    	Let $\rho = \sqrt{\frac{\lambda_1(\Omega)-\lambda_1}{\lambda_1(\Omega)\max\{\mu_1,\mu_2\}}}S$ and $\delta = \frac14 \frac{(\lambda_1(\Omega)-\lambda_1)^2}{\lambda_1(\Omega)^2\max\{\mu_1,\mu_2\}}S^2 +  \frac{\theta_1+\theta_2e^{-\frac{\lambda_{2}}{\theta_{2}}}}{2}|\Omega|$, we deduce that
    	\begin{align}
    		\mathcal{L}(u,v) \geq \frac14 \frac{(\lambda_1(\Omega)-\lambda_1)^2}{\lambda_1(\Omega)^2\max\{\mu_1,\mu_2\}}S^2 +  \frac{\theta_1+\theta_2e^{-\frac{\lambda_{2}}{\theta_{2}}}}{2}|\Omega| = \delta > 0.
    	\end{align}

    	(iii) If $(\lambda_1,\mu_1,\theta_1; \lambda_2,\mu_2,\theta_2) \in A_3$, then
    	\begin{align}
    		\mathcal{L}(u,v) \geq& \frac12|\nabla u|_2^2 - \frac{\mu_1}{4S^2}|\nabla u|_2^4 + \frac{\theta_1}{2}e^{-\frac{\lambda_{1}}{\theta_{1}}}|\Omega| + \frac12|\nabla v|_2^2 - \frac{\mu_2}{4S^2}|\nabla v|_2^4 + \frac{\theta_2}{2}e^{-\frac{\lambda_{2}}{\theta_{2}}}|\Omega| \\
    		\geq& \frac12(|\nabla u|_2^2+|\nabla v|_2^2) - \frac{1}{4S^2}\max\{\mu_1,\mu_2\}(|\nabla u|_2^2+|\nabla v|_2^2)^2 + \frac{\theta_1e^{-\frac{\lambda_{1}}{\theta_{1}}}+\theta_2e^{-\frac{\lambda_{2}}{\theta_{2}}}}{2}|\Omega|.
    	\end{align}
    	Let $\rho = \sqrt{\frac{1}{\max\{\mu_1,\mu_2\}}}S$ and $\delta = \frac14 \frac{1}{\max\{\mu_1,\mu_2\}}S^2 + \frac{\theta_1e^{-\frac{\lambda_{1}}{\theta_{1}}}+\theta_2e^{-\frac{\lambda_{2}}{\theta_{2}}}}{2}|\Omega|$, we deduce that
    	\begin{align}
    		\mathcal{L}(u,v) \geq \frac14 \frac{1}{\max\{\mu_1,\mu_2\}}S^2 + \frac{\theta_1e^{-\frac{\lambda_{1}}{\theta_{1}}}+\theta_2e^{-\frac{\lambda_{2}}{\theta_{2}}}}{2}|\Omega| = \delta > 0.
    	\end{align}

    	{\bf Case 2.} $\beta > 0$.
    	
    	Since $\beta\int_\Omega|u_n^+|^2|v_n^+|^2 \leq \frac12\beta\int_\Omega(\epsilon|u_n^+|^4+\frac{1}{\epsilon}|v_n^+|^4)$, we have
    	\begin{align}
    		\mathcal{L}(u,v)&\geq\frac{1}{2}\int_{\Omega} |\nabla u|^2-\frac{\lambda_1}{2}\int_{\Omega}|u^+|^2-\frac{\mu_1+\beta\epsilon}{4} \int_{\Omega}|u^+|^4-\frac{\theta_1}{2}\int_{\Omega}(u^+)^2(\log (u^+)^2-1)\\&+\frac{1}{2}\int_{\Omega} |\nabla v|^2-\frac{\lambda_2}{2}\int_{\Omega}|v^+|^2-\frac{\mu_2+\frac{\beta}{\epsilon}}{4} \int_{\Omega}|v^+|^4-\frac{\theta_2}{2}\int_{\Omega}(v^+)^2(\log (v^+)^2-1).
    	\end{align}
        Then similar to the Case 1, we can complete the proof.
    \end{proof}

    \begin{lemma} \label{energy level}
    	Under the hypotheses of Lemma \ref{lm}. Let $\mathcal{C}_\rho$ be given by Theorem \ref{Theorem-2}. Then $-\infty < \mathcal{C}_\rho < 0$.
    \end{lemma}

    \begin{proof}
    	When $|\nabla u|_2^2+|\nabla v|_2^2 < \rho^2$, it is not difficult to verify that $\mathcal{L}(u,v) > -\infty$, yielding that $\mathcal{C}_\rho > -\infty$. We show $\mathcal{C}_\rho < 0$ here. Note that $\theta_1 < 0$. Fix $(u,0)$ such that $|\nabla u|_2 < \rho$. Then we consider $\mathcal{L}(tu,0)$ with $t < 1$:
    	\begin{align}
    		&\mathcal{L}(tu,0) \\
    		=& t^2\left(\frac{1}{2}\int_{\Omega} |\nabla u|^2-\frac{\lambda_1}{2}\int_{\Omega}|u^+|^2-\frac{\mu_1}{4}t^2 \int_{\Omega}|u^+|^4-\frac{\theta_1}{2}\int_{\Omega}(u^+)^2(\log (u^+)^2-1)-\theta_1\log t\int_{\Omega}(u^+)^2\right).
    	\end{align}
    Choose $u$ such that $\int_{\Omega}(u^+)^2 > 0$. Then for small $t$ we have $\mathcal{L}(tu,0) < 0$. This completes the proof.
    \end{proof}

    \begin{lemma}[Boundedness of (PS) sequence] \label{boundedness}
    	Let $\{(u_n,v_n)\}$ be a (PS)$_c$ sequence and $\theta_i < 0, i = 1,2$. Then $\{(u_n,v_n)\}$ is bounded in $\mathcal{H}$.
    \end{lemma}

    \begin{proof}
    	Since $\mathcal{L}'(u_n,v_n) \to 0$, we have $\mathcal{L}'(u_n,v_n)(u_n,v_n) = o_n(|\nabla u_n|_2+|\nabla v_n|_2)$. Then for large $n$, we have
    	\begin{align}
    		& c + |\nabla u_n|_2+|\nabla v_n|_2 + 1 \\
    		&\geq \mathcal{L}(u_n,v_n) - \frac14 \mathcal{L}'(u_n,v_n)(u_n,v_n) \\
    		&= \frac{1}{4}|\nabla u_n|_2^2 -\frac{\theta_1}{4}\int_{\Omega}(u_n^+)^2 \log \sbr{e^{\frac{\lambda_{1}}{\theta_1}-2} (u_n^+)^2} +\frac{1}{4}|\nabla v_n|_2^2 -\frac{\theta_2}{4}\int_{\Omega}(v_n^+)^2 \log \sbr{e^{\frac{\lambda_{2}}{\theta_2}-2} (v_n^+)^2}.
    	\end{align}
        Since $\theta_1 < 0$, using $t\log t \geq -e^{-1}$, one gets
        \begin{align}
        	\frac{\theta_1}{4}\int_{\Omega}(u_n^+)^2 \log \sbr{e^{\frac{\lambda_{1}}{\theta_1}-2} (u_n^+)^2} &\leq \frac{\theta_1}{4}\int_{e^{\frac{\lambda_{1}}{\theta_1}-2} (u_n^+)^2\leq1}(u_n^+)^2 \log \sbr{e^{\frac{\lambda_{1}}{\theta_1}-2} (u_n^+)^2} \\
        	&\leq -\frac{\theta_1}{4}e^{2-\frac{\lambda_{1}}{\theta_1}}\int_{e^{\frac{\lambda_{1}}{\theta_1}-2} (u_n^+)^2\leq1} e^{-1} \\
        	&\leq -\frac{\theta_1}{4}e^{1-\frac{\lambda_{1}}{\theta_1}}|\Omega|.
        \end{align}
         Similarly, we have $$\frac{\theta_2}{4}\int_{\Omega}(v_n^+)^2 \log \sbr{e^{\frac{\lambda_{2}}{\theta_2}-2} (v_n^+)^2} \leq -\frac{\theta_2}{4}e^{1-\frac{\lambda_{2}}{\theta_2}}|\Omega|$$ since $\theta_2 < 0$. For $n$ large enough, we have
         \begin{align}
         	c + |\nabla u_n|_2+|\nabla v_n|_2 + 1 \geq \frac{1}{4}|\nabla u_n|_2^2 + \frac{1}{4}|\nabla v_n|_2^2 +\frac{\theta_1}{4}e^{1-\frac{\lambda_{1}}{\theta_1}}|\Omega| +\frac{\theta_2}{4}e^{1-\frac{\lambda_{2}}{\theta_2}}|\Omega|,
         \end{align}
         yielding to the boundedness of $\{(u_n,v_n)\}$ in $\mathcal{H}$.
    \end{proof}

    \begin{proof}[\bf Proof Theorem \ref{Theorem-2}]
    	By Lemma \ref{lm}, we can take a minimizing sequence $\lbr{(u_n,v_n)} \subset B_{\rho-\tau}$ for $\mathcal{C}_\rho$ with $\tau > 0$ small enough. By Ekeland's variational principle, we can assume that $\mathcal{L}'(u_n,v_n) \to 0$. By Lemma \ref{boundedness}, we can see that $\lbr{(u_n,v_n)}$ is bounded in $\mathcal{H} $. Hence, we may assume that
    	\begin{equation}
    		\sbr{u_n,v_n} \rightharpoonup (u,v) \text{ weakly in } \mathcal{H}.
    	\end{equation}
    	Passing to a subsequence, we may also assume that
    	\begin{equation}
    		\begin{aligned}
    			&u_n \rightharpoonup u, \quad v_n\rightharpoonup v \ \text{ weakly in } L^4(\Omega),\\
    			&u_n \to u, \quad v_n\to v \ \text{ strongly in } L^p(\Omega) \text{ for } 2\leq p<4,\\
    			&u_n \to u, \quad v_n \to v \ \text{ almost everywhere in } \Omega.
    		\end{aligned}
    	\end{equation}
    	By the weak lower semi-continuity of the norm, we see that $(u,v) \in B_\rho$. Similar to the proof of Theorem \ref{Theorem-1} (1) and (2), we have $\mathcal{L}^\prime (u,v)=0$ and $u \geq 0, v \geq 0$. Let $w_n=u_n-u$ and $z_n=v_n-v$. Then similar to the proof of Theorem \ref{Theorem-1} (1) and (2) one gets
    	\begin{equation}
    		\begin{aligned}
    			|\nabla w_n|_2^2=\mu_1|w_n^+|_4^4+o_n(1), \quad 	|\nabla z_n|_2^2=\mu_2|z_n^+|_4^4+o_n(1).
    		\end{aligned}
    	\end{equation}
    	\begin{equation} \label{eq energy}
    		\mathcal{L}(u_n,v_n)=\mathcal{L}(u,v)+\frac{1}{4}\int_{\Omega}|\nabla w_n|^2+\frac{1}{4}\int_{\Omega}|\nabla z_n|^2+o_n(1).
    	\end{equation}
    	Passing to a subsequence, we may assume that
    	\begin{equation}
    		\int_{\Omega}|\nabla w_n|^2=k_1+o_n(1), \quad \int_{\Omega}|\nabla z_n|^2=k_2+o_n(1).
    	\end{equation}	
    	Letting $n\to +\infty$ in \eqref{eq energy}, we have
    	\begin{equation}
    		\mathcal{C}_\rho \leq \mathcal{L}(u,v) \leq \mathcal{L}(u,v)+\frac{1}{4}k_1+\frac{1}{4}k_2=\lim_{n\to \infty} 	\mathcal{L}(u_n,v_n) = \mathcal{C}_\rho,
    	\end{equation}
        showing that $k_1 = k_2 = 0$. Hence, up to a subsequence we obtain
        \begin{equation}
        	\sbr{u_n,v_n} \to (u,v) \text{ strongly in } \mathcal{H}.
        \end{equation}
    	Since $\mathcal{L}(u,v) = \mathcal{C}_\rho < 0$ we have $(u,v) \neq (0,0)$.
    	
    	It remains to show that $(u,v)$ is not semi-trivial. We argue by contradiction. WLOG, assume that $u \equiv 0$, $v \neq 0$. We consider $\mathcal{L}(tw,v)$ where $t > 0$ is small enough such that $(tw,v) \in B_\rho$:
    	\begin{align}
    		\mathcal{L}(tw,v) = \mathcal{L}(0,v) &+ t^2(\frac{1}{2}\int_{\Omega} |\nabla w|^2-\frac{\lambda_1}{2}\int_{\Omega}|w^+|^2-\frac{\mu_1}{4}t^2 \int_{\Omega}|w^+|^4\\
    	    &-\frac{\theta_1}{2}\int_{\Omega}(w^+)^2(\log (u^+)^2-1)-\theta_1\log t\int_{\Omega}(w^+)^2 - \frac\beta 2\int_{\Omega}(w^+)^2(v^+)^2).
    	\end{align}
    	Choose $w$ such that $\int_{\Omega}(w^+)^2 > 0$. Then for small $t$ we have $\mathcal{L}(tw,v) < \mathcal{L}(0,v) = \mathcal{C}_\rho$, in a contradiction with the definition of $\mathcal{C}_\rho$. Hence $u \neq 0, v \neq 0$. By a similar argument as used in the proof of (1)-(2) in Theorem \ref{Theorem-1}, we can show that $u,v>0$ and $u,v \in C^2(\Omega)$.  This completes the proof.
    \end{proof}
	
	\section{Proof of Theorem \ref{least ene solu}} \label{sect l-e-s}
	
	\begin{lemma} \label{energy level of l-e-s}
		Under the hypotheses of Lemma \ref{lm}. We further assume that $2\min\{\theta_1,\theta_2\} \geq -\lambda_{1}(\Omega)$ or $\beta > 0$ or $\beta \in (-\sqrt{\mu_1\mu_2},0)$. Then $-\infty < \mathcal{C}_K < 0$.
	\end{lemma}

    \begin{proof}
    	Notice that $(u,v) \in K$ where $(u,v)$ is the solution given by Theorem \ref{Theorem-2}. Hence $\mathcal{C}_K \leq \mathcal{C}_\rho < 0$. Now we show that $\mathcal{C}_K > -\infty$.
    	
    	{\bf Case 1.} $2\min\{\theta_1,\theta_2\} \geq -\lambda_{1}(\Omega)$.
    	
    	For any $(u,v) \in K$ one gets
    	\begin{align}
    		\mathcal{L}(u,v)&=\mathcal{L}(u,v)-\frac{1}{4}\mathcal{L}^\prime(u,v)(u,v)\\
    		&= \frac{1}{4}|\nabla u|_2^2 -\frac{\theta_1}{4}\int_{\Omega}(u^+)^2 \log \sbr{e^{\frac{\lambda_{1}}{\theta_1}} (u^+)^2}+\frac{\theta_1}{2}|u^+|_2^2 \\
    		& +\frac{1}{4}|\nabla v|_2^2 -\frac{\theta_2}{4}\int_{\Omega}(v^+)^2 \log \sbr{e^{\frac{\lambda_{2}}{\theta_2}} (v^+)^2}+\frac{\theta_2}{2}|v^+|_2^2 .
    	\end{align}
        Since $\theta_1 < 0$, using $t\log t \geq -e^{-1}$, one gets
        \begin{align}
        	\frac{\theta_1}{4}\int_{\Omega}(u_n^+)^2 \log \sbr{e^{\frac{\lambda_{1}}{\theta_1}} (u_n^+)^2} &\leq \frac{\theta_1}{4}\int_{e^{\frac{\lambda_{1}}{\theta_1}} (u_n^+)^2\leq1}(u_n^+)^2 \log \sbr{e^{\frac{\lambda_{1}}{\theta_1}} (u_n^+)^2} \\
        	&\leq -\frac{\theta_1}{4}e^{-\frac{\lambda_{1}}{\theta_1}}\int_{e^{\frac{\lambda_{1}}{\theta_1}} (u_n^+)^2\leq1} e^{-1} \\
        	&\leq -\frac{\theta_1}{4}e^{-\frac{\lambda_{1}}{\theta_1}-1}|\Omega|.
        \end{align}
        Similarly, we have $$\frac{\theta_2}{4}\int_{\Omega}(v_n^+)^2 \log \sbr{e^{\frac{\lambda_{2}}{\theta_2}-2} (v_n^+)^2} \leq -\frac{\theta_2}{4}e^{-\frac{\lambda_{2}}{\theta_2}-1}|\Omega|$$ since $\theta_2 < 0$. Moreover, we have
        $$
        \frac{1}{4}|\nabla u|_2^2 + \frac{\theta_1}{2}|u^+|_2^2 \geq (\frac{1}{4}\lambda_{1}(\Omega) + \frac{\theta_1}{2})|u^+|_2^2 \geq 0
        $$
        and similarly $\frac{1}{4}|\nabla v|_2^2 + \frac{\theta_2}{2}|v^+|_2^2 \geq 0$. Hence,
        $$
        \mathcal{L}(u,v) \geq \frac{\theta_1}{4}e^{-\frac{\lambda_{1}}{\theta_1}-1}|\Omega| + \frac{\theta_2}{4}e^{-\frac{\lambda_{2}}{\theta_2}-1}|\Omega| > -\infty,
        $$
        showing that $\mathcal{C}_K > -\infty$.

        {\bf Case 2.} $\beta > 0$.

        For any $(u,v) \in K$ we obtain
        \begin{align}
        	&\mathcal{L}(u,v)-\frac{1}{2}\mathcal{L}^\prime(u,v)(u,v)\\
        	&=\frac{1}{4}\sbr{\mu_1|u^+|_4^4+\mu_2|v^+|_4^4+2\beta |u^+v^+|_2^2}+\frac{\theta_1}{2}|u^+|_2^2+\frac{\theta_2}{2}|v^+|_2^2 \\
        	&\geq \frac{1}{4}\sbr{\mu_1|u^+|_4^4+\mu_2|v^+|_4^4} + \frac{\theta_1}{2}|\Omega|^{\frac12}|u^+|_4^2 + \frac{\theta_2}{2}|\Omega|^{\frac12}|v^+|_4^2,
        \end{align}
        yielding that $\mathcal{C}_K > -\infty$.

        {\bf Case 3.} $\beta \in (-\sqrt{\mu_1\mu_2},0)$.

        In this case,
        $$
        2\beta |u^+v^+|_2^2 \geq \beta\sqrt{\frac{\mu_1}{\mu_2}}|u^+|_4^4 + \beta\sqrt{\frac{\mu_2}{\mu_1}}|v^+|_4^4.
        $$
        Then for any $(u,v) \in K$ we have
        \begin{align}
        	&\mathcal{L}(u,v)-\frac{1}{2}\mathcal{L}^\prime(u,v)(u,v)\\
        	&=\frac{1}{4}\sbr{(\mu_1+\beta\sqrt{\frac{\mu_1}{\mu_2}})|u^+|_4^4+(\mu_2+\beta\sqrt{\frac{\mu_2}{\mu_1}})|v^+|_4^4}+ \frac{\theta_1}{2}|\Omega|^{\frac12}|u^+|_4^2 + \frac{\theta_2}{2}|\Omega|^{\frac12}|v^+|_4^2.
        \end{align}
        From $\mu_1+\beta\sqrt{\frac{\mu_1}{\mu_2}} > 0$ and $\mu_2+\beta\sqrt{\frac{\mu_2}{\mu_1}} > 0$ we deduce that $\mathcal{C}_K > -\infty$.
    \end{proof}

    \begin{proof}[\bf Proof Theorem \ref{least ene solu}]
    	Take a minimizing sequence $\lbr{(u_n,v_n)} \subset K$ for $\mathcal{C}_K$. Then $\mathcal{L}'(u_n,v_n) = 0$. By Lemma \ref{boundedness}, we can see that $\lbr{(u_n,v_n)}$ is bounded in $\mathcal{H} $. Hence, we may assume that
    	\begin{equation}
    		\sbr{u_n,v_n} \rightharpoonup (u,v) \text{ weakly in } \mathcal{H}.
    	\end{equation}
    	Passing to subsequence, we may also assume that
    	\begin{equation}
    		\begin{aligned}
    			&u_n \rightharpoonup u, \quad v_n\rightharpoonup v \ \text{ weakly in } L^4(\Omega),\\
    			&u_n \to u, \quad v_n\to v \ \text{ strongly in } L^p(\Omega) \text{ for } 2\leq p<4,\\
    			&u_n \to u, \quad v_n \to v \ \text{ almost everywhere in } \Omega.
    		\end{aligned}
    	\end{equation}
    	Similar to the proof of Theorem \ref{Theorem-1} (1) and (2), we have $\mathcal{L}^\prime (u,v)=0$ and $u \geq 0, v \geq 0$. Let $w_n=u_n-u$ and $z_n=v_n-v$. Then similar to the proof of Theorem \ref{Theorem-1} (1) and (2) one gets
    	\begin{equation}
    		\begin{aligned}
    			|\nabla w_n|_2^2=\mu_1|w_n^+|_4^4+o_n(1), \quad 	|\nabla z_n|_2^2=\mu_2|z_n^+|_4^4+o_n(1).
    		\end{aligned}
    	\end{equation}
    	\begin{equation} \label{eq energy2}
    		\mathcal{L}(u_n,v_n)=\mathcal{L}(u,v)+\frac{1}{4}\int_{\Omega}|\nabla w_n|^2+\frac{1}{4}\int_{\Omega}|\nabla z_n|^2+o_n(1).
    	\end{equation}
    	Passing to subsequence, we may assume that
    	\begin{equation}
    		\int_{\Omega}|\nabla w_n|^2=k_1+o_n(1), \quad \int_{\Omega}|\nabla z_n|^2=k_2+o_n(1).
    	\end{equation}	
    	Letting $n\to +\infty$ in \eqref{eq energy2}, we have
    	\begin{equation}
    		\mathcal{C}_K \leq \mathcal{L}(u,v) \leq \mathcal{L}(u,v)+\frac{1}{4}k_1+\frac{1}{4}k_2=\lim_{n\to \infty} 	\mathcal{L}(u_n,v_n) = \mathcal{C}_K,
    	\end{equation}
    	showing that $k_1 = k_2 = 0$. Hence, up to a subsequence we obtain
    	\begin{equation}
    		\sbr{u_n,v_n} \to (u,v) \text{ strongly in } \mathcal{H}.
    	\end{equation}
    	Since $\mathcal{L}(u,v) = \mathcal{C}_K < 0$ we have $(u,v) \neq (0,0)$.
    \end{proof}

	\section{Proof of Theorem \ref{Theorem-5}} \label{Sect6}
	
First, we give two energy estimates, which play important roles in the proof of Theorem \ref{Theorem-5}.
	\begin{proposition}\label{Lemma energy estimate-4}
		Assume that  $\beta \neq 0$, then we have
		\begin{equation}
			\mathcal{C}_M<\min\lbr{\frac{1}{4}\mu_1^{-1}\mathcal{S}^2, \frac{1}{4}\mu_2^{-1}\mathcal{S}^2}<\frac{1}{4}(\mu_1^{-1}+\mu_2^{-1})\mathcal{S}^2,
		\end{equation}
	where $\mathcal{C}_M$ is defined in \eqref{defi of B neqgative}.
	\end{proposition}
	\begin{proof} 
		Without loss of generality, we only prove the inequality 	$\mathcal{C}_M<\frac{1}{4}\mu_1^{-1}\mathcal{S}^2$.  By the definition of $\mathcal{C}_M$,  we can easily check   that
		$$	\mathcal{C}_M\leq \inf_{(u,v)\in \mathcal{H}\setminus \lbr{(0,0)}} \max_{t>0} \ \mathcal{L}(tu,tv)\leq \max_{t >0} \mathcal{L}(tv_{\varepsilon},0),$$
		where $v_\varepsilon$ is defined in \eqref{f64} with $R=R_{\max}$.  Then the proof  follows directly from that in \cite[Lemma 3.5]{Deng-He-Pan=Arxiv=2022}  and  we omit the details.
	\end{proof}
	
	\begin{proposition}\label{Lemma energy estimate-5}
		Assume that   $\beta\in \sbr{-\infty,0}\cup \sbr{ 0,\min\lbr{\mu_1,\mu_2}}  \cup  \sbr{\max\lbr{\mu_1,\mu_2},+\infty} $ and
		$$	\frac{32 e^{{\lambda_{i}}/{\theta_{i}}}}{\mu_iR_{\max}^2}<1, \text{ with } \ \ R_{\max}:=\sup\lbr{r>0:\exists x\in \Omega, \text{ s.t. } B_r(x)\subset \Omega}, ~i=1,2,$$
		then we have
		\begin{equation}
			\mathcal{C}_M< \mathcal{A},
		\end{equation}
		where  $\mathcal{A}$ is given by \eqref{defi of A}.
	\end{proposition}
	\begin{proof}
		
		Firstly, for the case $\beta <0$,    the conclusion  follows directly from Proposition \ref{Lemma energy estimate-4}  and formula \eqref{f13}.  Now we turn to the  proof for the case  $\beta\in  \sbr{ 0,\min\lbr{\mu_1,\mu_2}}  \cup  \sbr{\max\lbr{\mu_1,\mu_2},+\infty}$.
		
		Without loss of generality, we assume that there exist $y_0\in \Omega$ such that $y_0$ is the geometric center of $\Omega$, i.e. $R_{\max}=\mathrm{dist}(y_0,\partial \Omega)$. For simplicity, we still use the notations in the proof of  Proposition \ref{Lemma energy estimate-2} and take $R=R_{\max}$. Let  $G(t):=\mathcal{L}(tw_{\varepsilon},tz_{\varepsilon})$ (recall that $w_{\varepsilon}, z_{\varepsilon}$ is defined in\eqref{f19}), then a direct computation gives that
		\begin{equation}
			\begin{aligned}
				G(t)& =\mathcal{L}(t 	w_\varepsilon,t z_{\varepsilon})\\
				&=\frac{1}{2}t^2\int_{\Omega}\sbr{|\nabla w_\varepsilon|^2+ |\nabla z_{\varepsilon}|^2}-\frac{1}{4}t^4\int_{\Omega}\sbr{\mu_1|w_\varepsilon|^4+2\beta |w_\varepsilon|^2|z_{\varepsilon}|^2+\mu_2|z_\varepsilon|^4}\\
				&-\frac{\lambda_1-\theta_1}{2}t^2\int_{\Omega}|w_{\varepsilon}|^2 -\frac{\theta_1}{2}t^2\log t^2\int_{\Omega} |w_{\varepsilon}|^2-\frac{\theta_1}{2}t^2\int_{\Omega} w_{\varepsilon}^2\log w_{\varepsilon}^2\\
				&-\frac{\lambda_2-\theta_2}{2}t^2\int_{\Omega}|z_{\varepsilon}|^2 -\frac{\theta_2}{2}t^2\log t^2\int_{\Omega} |z_{\varepsilon}|^2-\frac{\theta_2}{2}t^2\int_{\Omega} z_{\varepsilon}^2\log z_{\varepsilon}^2\\
				&	=:I_1(t)+I_2(t)+I_3(t).
			\end{aligned}
		\end{equation}
		It is easy to see that $G(0)=0$ and $\lim_{t \to +\infty} G(t)=-\infty$, then we can find $t_{\varepsilon}\in \sbr{0,+\infty}$ such that
		\begin{equation}
			\max_{t >0}	\mathcal{L}(tw_{\varepsilon},tz_{\varepsilon})	=\max_{t >0} G(t)=G(t_{\varepsilon}).
		\end{equation}
		So, we have
		\begin{equation}
			\begin{aligned}
				\int_{\Omega}\sbr{|\nabla w_\varepsilon|^2+ |\nabla z_{\varepsilon}|^2}&=\int_{\Omega}\sbr{\lambda_{1}|w_\varepsilon|^2+\lambda_{2}|z_\varepsilon|^2} +\theta_{1}\int_{\Omega} w_{\varepsilon}^2\log w_{\varepsilon}^2+\theta_{2}\int_{\Omega} z_{\varepsilon}^2\log z_{\varepsilon}^2\\
				&+ t_\varepsilon^2\int_{\Omega}\sbr{\mu_1|w_\varepsilon|^4+2\beta |w_\varepsilon|^2|z_{\varepsilon}|^2+\mu_2|z_\varepsilon|^4}+\log t_\varepsilon^2
				\int_{\Omega}\sbr{\theta_{1}|w_\varepsilon|^2+\theta_{2}|z_\varepsilon|^2}.
			\end{aligned}
		\end{equation}
		A standard argument implies that $t_{\varepsilon}$ is bounded from below and above for $\varepsilon$ small enough. Moreover,  we can see from  \eqref{f14} and \eqref{es-7} that $t_{\varepsilon}\to 1$ as $\varepsilon\to 0^+$. Therefore,
		\begin{equation}
			\log t_\varepsilon^2 \int_{\Omega}\theta_{1}|w_\varepsilon|^2=o(\varepsilon^2|\log \varepsilon|), \quad  \log t_\varepsilon^2 \int_{\Omega}\theta_{2}|z_\varepsilon|^2=o(\varepsilon^2|\log \varepsilon|).
		\end{equation}
		By using \eqref{f14} and \eqref{es-7} once more, we have
		\begin{equation}
			\begin{aligned}
				I_1(t_{\varepsilon})&=\mbr{\frac{1}{2}(k+l)t_{\varepsilon}^2-\frac{1}{4}(\mu_1k^2+2\beta kl+\mu_2l^2)t_{\varepsilon}^4}(\mathcal{S}^2+O(\varepsilon^2))\\
				&\leq \frac{1}{4}(k+l)\mathcal{S}^2+O(\varepsilon^2)=\mathcal{A}+O(\varepsilon^2).
			\end{aligned}
		\end{equation}
		On the other hand, since $\theta_{1}<0$, by \eqref{es-4} we have
		\begin{equation}
			\begin{aligned}
				I_2(t_\varepsilon)&=-\frac{\lambda_1-\theta_1}{2}t_{\varepsilon}^2\int_{\Omega}|w_{\varepsilon}|^2 -\frac{\theta_1}{2}t_{\varepsilon}^2\log t_{\varepsilon}^2\int_{\Omega} |w_{\varepsilon}|^2-\frac{\theta_1}{2}t_{\varepsilon}^2\int_{\Omega} w_{\varepsilon}^2\log w_{\varepsilon}^2\\
				&=-\frac{k(\lambda_1-\theta_1)+\theta_1 k\log  k}{2}t_{\varepsilon}^2\int_{\Omega}|v_{\varepsilon}|^2-\frac{\theta_1k}{2}t_{\varepsilon}^2\int_{\Omega} v_{\varepsilon}^2\log v_{\varepsilon}^2+o(\varepsilon^2|\log \varepsilon|)\\
				&\leq -\frac{t_{\varepsilon}^2}{2}\mbr{\sbr{k\lambda_1-\theta_1k+\theta_1 k\log  k} 8\omega_4\varepsilon^2|\log \varepsilon| + 8\theta_1k \log \sbr{\cfrac{8e(\varepsilon^2+4R^2)}{(\varepsilon^2+R^2)^2}}\omega_4 \varepsilon^2|\log \varepsilon|}+o(\varepsilon^2|\log \varepsilon|)\\
				&\leq -4t_{\varepsilon}^2\log\sbr{\cfrac{8ke^{\lambda_{1}/\theta_1}(\varepsilon^2+4R^2)}{(\varepsilon^2+R^2)^2}}\theta_1k\omega_4 \varepsilon^2|\log \varepsilon|+o(\varepsilon^2|\log \varepsilon|)\\
				&\leq -4t_{\varepsilon}^2\log\sbr{\cfrac{32ke^{\lambda_{1}/\theta_1}}{R^2}}\theta_1k\omega_4 \varepsilon^2|\log \varepsilon|+o(\varepsilon^2|\log \varepsilon|).
			\end{aligned}
		\end{equation}
		Similarly, we have
		\begin{equation}
			\begin{aligned}
				I_3(t_\varepsilon)\leq -4t_{\varepsilon}^2\log\sbr{\cfrac{32le^{\lambda_{2}/\theta_2}}{R^2}}\theta_2k\omega_4 \varepsilon^2|\log \varepsilon|+o(\varepsilon^2|\log \varepsilon|).
			\end{aligned}
		\end{equation}
		Recall that $(k,l)$ is a solution of \eqref{f14}, it is easy to see from $0<\beta<\min\lbr{\mu_1,\mu_2}$ and $\beta>\max \lbr{\mu_1,\mu_2}$ that $0<k<\frac{1}{\mu_1}$ and $0<l<\frac{1}{\mu_2}$.
		Then we can deduce from the assumptions that $\frac{32ke^{\lambda_{1}/\theta_1}}{R^2}<1$, $\frac{32le^{\lambda_{2}/\theta_2}}{R^2}<1$, and
		\begin{equation}
			I_2(t_\varepsilon)\leq -C\log\sbr{\cfrac{32ke^{\lambda_{1}/\theta_1}}{  R^2}}\theta_1k\omega_4 \varepsilon^2|\log \varepsilon|+o(\varepsilon^2|\log \varepsilon|),
		\end{equation}
		\begin{equation}
			I_3(t_\varepsilon)\leq -C\log\sbr{\cfrac{32le^{\lambda_{2}/\theta_2}}{R^2}}\theta_2k\omega_4 \varepsilon^2|\log \varepsilon|+o(\varepsilon^2|\log \varepsilon|).
		\end{equation}
		So,
		\begin{equation}
			\mathcal{C}_M\leq \max_{t >0} \mathcal{L}(t_{\varepsilon}	w_\varepsilon,t_{2,\varepsilon}z_{\varepsilon})= I_1(t_{\varepsilon})+I_2(t_{\varepsilon})+I_3(t_{\varepsilon})<\mathcal{A},
		\end{equation}
		for $\varepsilon$ small enough. This completes the proof.
	\end{proof}

In the end of this section, we present the proof of Theorem \ref{Theorem-5}.

\begin{proof}[\bf Proof of Theorem \ref{Theorem-5}.]  Under the hypotheses of Theorem \ref{Theorem-5}. By Lemma \ref{lm}, we can easily see that the functional $\mathcal{L}$ has  a  mountain pass structure. Then by the mountain pass theorem, there exists a sequence $\lbr{\sbr{u_n,v_n}}\in \mathcal{H}$ such that
	\begin{equation} \label{f65}
		\lim_{n\to \infty} \mathcal{L}\sbr{u_n,v_n}=\mathcal{C}_M>0,\quad \text{ and }\quad 	\lim_{n\to \infty} \mathcal{L}^\prime\sbr{u_n,v_n}=0.
	\end{equation}
 By Lemma \ref{boundedness}, we can see that $\lbr{(u_n,v_n)}$ is bounded in $\mathcal{H} $.  So   there exists  $(u,v)\in \mathcal{H} $, such that up to subsequence,
	\begin{equation}
		\begin{aligned}
			&u_n \rightharpoonup u, \quad  v_n \rightharpoonup v, \text{ weakly in } H_0^1(\Omega),\\
			&u_n \rightharpoonup u, \quad v_n\rightharpoonup v \ \text{ weakly in } L^4(\Omega),\\
			&u_n \to u, \quad v_n\to v \ \text{ strongly in } L^p(\Omega) \text{ for } 2\leq p<4,\\
			&u_n \to u, \quad v_n \to v \ \text{ almost everywhere in } \Omega.
		\end{aligned}
	\end{equation}
	It is easy to see that  $ \mathcal{L}^\prime\sbr{u_n,v_n}(\varphi_1,\varphi_2)
	\to 0$ as $n\to \infty$ for any  $\varphi_1,\varphi_2 \in H_0^1(\Omega)$. So, $(u,v)$ is a weak solution of the system
	\begin{equation}
		\begin{cases}
			-\Delta u=\lambda_{1}u^++
			\mu_1|u^+|^{2}u^++\beta |v^+|^{2}u^++\theta_1 u^+\log (u^+)^2, & \quad x\in \Omega,\\
			-\Delta v=\lambda_{2}v^++
			\mu_2|v^+|^{2}v^++\beta |u^+|^{2}v^++\theta_2 v^+\log (v^+)^2,  &\quad x\in \Omega.
		\end{cases}
	\end{equation}
	Set $w_n=u_n-u$ and $z_n=v_n-v$, by a similar argument as used in the proof of Theorem \ref{Theorem-1}, we have
	\begin{equation} \label{f66}
		\begin{aligned}
			|\nabla w_n|_2^2=\mu_1|w_n^+|_4^4+\beta|w_n^+z_n^+|_2^2+o_n(1), \quad 	|\nabla z_n|_2^2=\mu_2|z_n^+|_4^4+\beta|w_n^+z_n^+|_2^2+o_n(1).
		\end{aligned}
	\end{equation}
	and
	\begin{equation}\label{f67}
		\mathcal{L}(u,v)+\frac{1}{4}k_1+\frac{1}{4}k_2=	\lim_{n\to \infty} \mathcal{L}(u_n,v_n)=\mathcal{C}_M,
	\end{equation}
	where  	$k_1:=\lim_{n\to \infty}\int_{\Omega}|\nabla w_n|^2\geq 0$ and $ k_2:=\lim_{n\to \infty}\int_{\Omega}|\nabla z_n|^2\geq 0$.
	
	Now we claim that $(u,v)\neq (0,0)$. Supposing by contradiction, if  $(u,v)=(0,0)$, then $\mathcal{L}(u,v)=0$. We can see from \eqref{f65} that  $k_1+k_2>0$. By contradiction,  if $k_1>0$, $k_2=0$, then we have $|w_n^+z_n^+|_2^2=o_n(1)$.  By \eqref{f66}, we have
	\begin{equation}
		\int_{\Omega}|\nabla w_n|^2 =\mu_1|w_n^+|_4^4+o_n(1)\leq \mu_1\mathcal{S}^{-2}\sbr{\int_{\Omega}|\nabla w_n|^2 }^2.
	\end{equation}
	Letting $n\to \infty$, we have  $k_1\geq \mu_1^{-1}\mathcal{S}^{2}$. Hence, by \eqref{f67}, we know that $\mathcal{C}_M\geq \frac{1}{4}\mu_1^{-1}\mathcal{S}^{2}$, which contradicts to Proposition \ref{Lemma energy estimate-4}. Similarly, the case $k_1=0$, $k_2>0$ also implies a contradiction.  So both $k_1$ and $k_2$ are positive.
	Following the same proof as Case 1 in the proof of (1)-(2) of Theorem \ref{Theorem-1}, we have
	\begin{equation}
		\mathcal{C}_M\geq \mathcal{A},
	\end{equation}
	a contradiction with Proposition \ref{Lemma energy estimate-5}. Hence, $(u,v)\in \mathcal{H}\setminus\lbr{(0,0)}$.  Finally, by a similar argument as used in that of  the proof of (1)-(2) of Theorem \ref{Theorem-1}, we have $u\geq0, v\geq 0$, and $u,v \in C^2(\Omega)$. This completes the proof.
\end{proof}

	\section{The nonexistence results}	\label{Sect5}
	
	\begin{proof}[\bf Proof of Theorem \ref{Theorem-3}]
	
	WLOG, we assume $(\lambda_{1}, \mu_{1}, \theta_{1}) \in \varSigma_2$. The proof of $(\lambda_{2}, \mu_{2}, \theta_{2}) \in \varSigma_2$ is similar. We argue by contradiction. Assume that the system \eqref{System1}  has a positive solution $(u,v)$. Multiply the equation for $u$ in \eqref{System1} by $e_1$ where $e_1(x)$ is the first eigenfunction corresponding to $\lambda_1(\Omega)$, and integrate over $\Omega$, which yields
	\begin{equation} \label{eq5.1}
		\lambda_1(\Omega)\int_{\Omega}ue_1 = \int_{\Omega}\nabla u \nabla e_1 = \int_{\Omega}(\lambda_{1} + \mu_1|u|^2 + \theta_{1}\log u^2 + \beta |v|^2)ue_1.
	\end{equation}
    Let $g(t) = \mu_1 t + \theta_1 \log t + \lambda_{1}, t > 0$. Since $g'(t) = \mu_1 + \frac{\theta_1}{t}$, we obtain
    $$
    g(t) \geq g(-\frac{\theta_1}{\mu_1}) = \lambda_{1} + |\theta_1| + \theta_1\log|\theta_1| - \theta_1\log\mu_1.
    $$
    Noticing $(\lambda_{1}, \mu_{1}, \theta_{1}) \in \varSigma_2$, one gets that $g(t) \geq \lambda_{1}(\Omega)$ for all $t > 0$. Then \eqref{eq5.1} yields that
    $$
    \lambda_1(\Omega)\int_{\Omega}ue_1 \geq \int_{\Omega}(\lambda_{1}(\Omega) + \beta |v|^2)ue_1,
    $$
    i.e.
    $$
    \beta\int_{\Omega}|v|^2ue_1 \leq 0.
    $$
    This is a contradiction since $\beta > 0$ and $\int_{\Omega}|v|^2ue_1 > 0$. The proof is complete.
	\end{proof}

	\begin{proof}[\bf Proof of Theorem \ref{Theorem-4}]
	Assume that the system \eqref{System1}  has a positive solution $(u,v)$. Multiply the equation for $u$ in \eqref{System1} by $v$, the equation for $v$ by $u$, and integrate over $\Omega$, which yields
	\begin{equation} \label{eq5.2}
		\int_{\Omega}\left((\lambda_{2}-\lambda_{1})+(\mu_2-\beta)|v|^2-(\mu_1-\beta)|u|^2+\theta_2\log v^2-\theta_1\log u^2\right)uv = 0.
	\end{equation}	
    Let $g(s,t) = (\mu_2-\beta)s-(\mu_1-\beta)t+\theta_2\log s-\theta_1\log t, s,t > 0$. Since $g'_s(s,t) = \mu_2-\beta+\frac{\theta_2}{s}, g'_t(s,t) = \beta - \mu_1 - \frac{\theta_1}{t}$, we obtain
    $$
    g(s,t) \geq g(-\frac{\theta_2}{\mu_2-\beta},\frac{\theta_1}{\beta-\mu_1}) =-\theta_{1} \log \sbr{\frac{\theta_{1}}{\beta-\mu_{1}}}+\theta_{2} \log \sbr{\frac{\theta_{2}}{\beta-\mu_{2}}}.
    $$
    Under the conditions of Theorem \ref{Theorem-4}, one gets that $g(s,t) > \lambda_{1}-\lambda_{2}$ for all $s,t > 0$. Then \eqref{eq5.2} provides a contradiction. The proof is complete.
		
	\end{proof}

    \appendix

    \section{Single equation with $\theta < 0$} \label{single equation}
    In this appendix, we consider the following equation:	
    	\begin{equation} \label{equation of u}
    		\begin{cases}
    			-\Delta u=\lambda u + \mu |u|^{2}u+\theta u\log u^2, & \quad x\in \Omega\\
    			u=0, &\quad  x \in \partial \Omega,
    		\end{cases}
    	\end{equation}
        where $\Omega \subset \mathbb{R}^4$ is bounded and $\mu > 0, \theta < 0$.
    	We define the associated modified energy functional
    	\begin{equation}
    		J(u)=\frac{1}{2} \int_{\Omega}|\nabla u|^2-\frac{\lambda}{2} \int_{\Omega}|u^+|^2-\frac{\mu}{4}\int_{\Omega} |u^+|^4-\frac{\theta}{2} \int_{\Omega} (u^+)^2\sbr{\log (u^+)^2-1}.
    	\end{equation}
    	Set
    	$$
    	\varSigma_3 := \left\{(\lambda,\mu,\theta): \lambda \in [0,\lambda_{1}(\Omega)), \mu > 0, \theta < 0, \frac{(\lambda_1(\Omega)-\lambda)^2}{\lambda_1(\Omega)^2\mu}S^2 + 2\theta|\Omega| > 0\right\},
    	$$
    	$$
    	\varSigma_4 := \left\{(\lambda,\mu,\theta): \lambda \in \mathbb{R}, \mu > 0, \theta < 0, \mu^{-1}S^2 +  2\theta e^{-\frac{\lambda}{\theta}}|\Omega| > 0\right\}.
    	$$
    	
    	\begin{lemma} \label{lem lm}
    		Assume that $(\lambda,\mu,\theta) \in \varSigma_3\cup\varSigma_4$.
    		Then there exist $\delta, \rho > 0$ such that $J(u) \geq \delta$ for all $|\nabla u|_2 = \rho$.
    	\end{lemma}
    	
    	\begin{proof}
%
%
%
The proof can be found in \cite{Deng-He-Pan=Arxiv=2022}, so we omit it.
    	\end{proof}

    	\begin{theorem}[Existence of a local minima] \label{l-m}
    		Assume that $(\lambda,\mu,\theta) \in \varSigma_3\cup\varSigma_4$. Define
    		$$
    		\tilde{c}_\rho := \inf_{|\nabla u|_2 < \rho}J(u),
    		$$
    		where $\rho$ is given by Lemma \ref{lem lm}.
    		Then equation \eqref{equation of u} has a positive solution $u$ such that $J(u) = \tilde{c}_\rho$.
    	\end{theorem}

    	\begin{proof}
    		Similar to Lemma \ref{energy level} we obtain that $-\infty < \tilde{c}_\rho < 0$. By Lemma \ref{lem lm}, we can take a minimizing sequence $\lbr{u_n}$ for $\tilde{c}_\rho$ with $|\nabla u_n|_2 < \rho - \tau$ and $\tau > 0$ small enough. By Ekeland's variational principle, we can assume that $J'(u_n) \to 0$. Similar to Lemma \ref{boundedness}, we can see that $\lbr{u_n}$ is bounded in $H_0^1(\Omega)$. Hence, we may assume that
    		\begin{equation}
    			u_n \rightharpoonup u \text{ weakly in } H_0^1(\Omega).
    		\end{equation}
    		Passing to subsequence, we may also assume that
    		\begin{equation}
    			\begin{aligned}
    				&u_n \rightharpoonup u \ \text{ weakly in } L^4(\Omega),\\
    				&u_n \to u \ \text{ strongly in } L^p(\Omega) \text{ for } 2\leq p<4,\\
    				&u_n \to u \ \text{ almost everywhere in } \Omega.
    			\end{aligned}
    		\end{equation}
    		By the weak lower semi-continuity of the norm, we see that $|\nabla u|_2 < \rho$. Similar to the proof of Theorem \ref{Theorem-1} (1) and (2), we have $J^\prime (u)=0$ and $u \geq 0$. Let $w_n=u_n-u$. Then similar to the proof of Theorem \ref{Theorem-1} (1) and (2) one gets
    		\begin{equation}
    			\begin{aligned}
    				|\nabla w_n|_2^2=\mu|w_n^+|_4^4+o_n(1).
    			\end{aligned}
    		\end{equation}
    		\begin{equation} \label{eq energy 3}
    			J(u_n)=J(u)+\frac{1}{4}\int_{\Omega}|\nabla w_n|^2+o_n(1).
    		\end{equation}
    		Passing to subsequence, we may assume that
    		\begin{equation}
    			\int_{\Omega}|\nabla w_n|^2=k+o_n(1).
    		\end{equation}	
    		Letting $n\to +\infty$ in \eqref{eq energy 3}, we have
    		\begin{equation}
    			\tilde{c}_\rho \leq J(u) \leq J(u)+\frac{1}{4}k=\lim_{n\to \infty} J(u_n) = \tilde{c}_\rho,
    		\end{equation}
    		showing that $k=0$. Hence, up to a subsequence we obtain
    		\begin{equation}
    			u_n \to u \text{ strongly in } H_0^1(\Omega).
    		\end{equation}
    		Since $J(u) = \tilde{c}_\rho < 0$ we have $u \neq 0$. Then by a similar argument as used in the proof of (1)-(2) in Theorem \ref{Theorem-1}, we can show that $u>0$ and $u \in C^2(\Omega)$.  This completes the proof.
    	\end{proof}
    	
    	\begin{theorem}[Existence of the least energy solution] \label{l-e-s}
    		Assume that $(\lambda,\mu,\theta) \in \varSigma_3\cup\varSigma_4$. Define
    		$$
    		\tilde{c}_\mathcal{K} := \inf_{u\in \mathcal{K}}J(u)
    		$$
    		where
    		$$\mathcal{K} = \{u \in H^1_0(\Omega): J'(u) = 0\}.$$
    		Then equation \eqref{equation of u} has a positive least energy solution $u$ such that $J(u) = \tilde{c}_\mathcal{K}$.
    	\end{theorem}

    	\begin{proof}
    		Similar to the case 2 in the proof of Lemma \ref{energy level of l-e-s} we have $-\infty < \tilde{c}_\mathcal{K} < 0$. Take a minimizing sequence $\lbr{u_n} \subset \mathcal{K}$ for $\tilde{c}_\mathcal{K}$. Then $J'(u_n) = 0$. SImilar to Lemma \ref{boundedness}, we can see that $\lbr{u_n}$ is bounded in $H_0^1(\Omega)$. Hence, we may assume that
    		\begin{equation}
    			u_n \rightharpoonup u \text{ weakly in } H_0^1(\Omega).
    		\end{equation}
    		Passing to subsequence, we may also assume that
    		\begin{equation}
    			\begin{aligned}
    				&u_n \rightharpoonup u \ \text{ weakly in } L^4(\Omega),\\
    				&u_n \to u \ \text{ strongly in } L^p(\Omega) \text{ for } 2\leq p<4,\\
    				&u_n \to u \ \text{ almost everywhere in } \Omega.
    			\end{aligned}
    		\end{equation}
    		Similar to the proof of Theorem \ref{Theorem-1} (1) and (2), we have $J^\prime (u)=0$ and $u \geq 0$. Let $w_n=u_n-u$. Then similar to the proof of Theorem \ref{Theorem-1} (1) and (2) one gets
    		\begin{equation}
    			\begin{aligned}
    				|\nabla w_n|_2^2=\mu|w_n^+|_4^4+o_n(1).
    			\end{aligned}
    		\end{equation}
    		\begin{equation} \label{eq energy 4}
    			J(u_n)=J(u)+\frac{1}{4}\int_{\Omega}|\nabla w_n|^2+o_n(1).
    		\end{equation}
    		Passing to subsequence, we may assume that
    		\begin{equation}
    			\int_{\Omega}|\nabla w_n|^2=k+o_n(1).
    		\end{equation}	
    		Letting $n\to +\infty$ in \eqref{eq energy 4}, we have
    		\begin{equation}
    			\tilde{c}_\mathcal{K} \leq J(u) \leq J(u)+\frac{1}{4}k=\lim_{n\to \infty} J(u_n) = \tilde{c}_\mathcal{K},
    		\end{equation}
    		showing that $k = 0$. Hence, up to a subsequence we obtain
    		\begin{equation}
    			u_n \to u \text{ strongly in } H_0^1(\Omega).
    		\end{equation}
    		Since $J(u) = \tilde{c}_\mathcal{K} < 0$ we have $u \neq 0$. Then by a similar argument as used in the proof of (1)-(2) in Theorem \ref{Theorem-1}, we can show that $u>0$ and $u \in C^2(\Omega)$.  This completes the proof.
    	\end{proof}
    	
    	\begin{remark}
    	{\rm 	In \cite[Theorem 1.3]{Deng-He-Pan=Arxiv=2022}, the authors showed that \eqref{equation of u} possesses a positive solution when $(\lambda,\mu,\theta) \in \varSigma_3 \cup \varSigma_4$ with $\frac{32e^{\frac{\lambda_i}{\theta_i}}}{\rho_{max}^2} < 1$, where $$\rho_{max} := \sup\{r > 0: \exists x \in \Omega \ s.t. \ B(x,r) \subset \Omega\}.$$ But they don't know the type of the solution and its energy level. In our Theorems \ref{l-m} and \ref{l-e-s}, we remove the condition $\frac{32e^{\frac{\lambda_i}{\theta_i}}}{\rho_{max}^2} < 1$. Furthermore, we give the type of the positive solution (a local minimum or a least energy solution) and show that its energy level is negative.}
    	\end{remark}

        It is easy to see that there is a mountain pass geometry since Theorem \ref{l-m} shows the existence of a local minimum $u$ and $J(tw) \to -\infty$ as $t \to \infty$ when $w^+ \neq 0$. Set
        \begin{equation}
        	\tilde{c}_{M} := \inf_{\gamma \in \tilde{\Gamma}}\sup_{t \in [0,1]}J(\gamma(t)),
        \end{equation}
        where
        $$
        \tilde{\Gamma} := \{\gamma \in C([0,1],H_0^1(\Omega)): \gamma(0) = u, J(\gamma(1)) < J(u)\},
        $$
        $u$ is the local minimum given by Theorem \ref{l-m}.

        {\bf Conjecture 2:} Equation \eqref{equation of u} possesses a positive mountain pass solution at level $\tilde{c}_{M} > 0$.

        \begin{remark}
        	If $\tilde{c}_M$ has the following estimate:
        	$$
        	\tilde{c}_M < \tilde{c}_\mathcal{K} + \frac14\mu^{-1}S^2,
        	$$
        	then Conjecture 2 holds true.
        \end{remark}


\begin{thebibliography}{1}
		\bibitem{Alfaro=DPDE=2017}	M. Alfaro, R. Carles,{ \it Superexponential growth or decay in the heat equation with a logarithmic nonlinearity.} 	Dyn. Partial Differ. Equ. 14 (2017), 343--358.
		
		\bibitem{Aubin=JDG=1976} T. Aubin, {\it Probl\`emes isop$\acute{e}$rim$\acute{e}$triques et espaces de Sobolev. } Journal of Differential Geometry. 11 (1976), 573--598.
		
		
		\bibitem{Ambrosotti=JFA=1973}  A. Ambrosetti,  P. H. Rabinowitz, { \it Dual variational methods in critical point theory and applications. }  Journal of Functional Analysis.  14 (1973), 349--381.
		
		\bibitem{Akhmediev1999}   N. Akhmediev, A. Ankiewicz, {\it Partially coherent solitons on a finite background.} Phys. Rev. Lett. 82 (1999), 2661--2664.
		
		
		\bibitem{Bialynicki=1975} I. Bialynicki-Birula, J. Mycielski, {\it Wave equations with logarithmic nonlinearities.} Bull. Acad. Polon.
		Sci.23 (1975), 461--466.
		
		\bibitem{Bialynicki=1976} I. Bialynicki-Birula, J. Mycielski, {\it Nonlinear wave mechanics.} Ann. Physics,100 (1976), 62--93.
		
		\bibitem{Brezis-Nirenberg1983} H. Br\'ezis, L. Nirenberg, {\it Positive solutions of nonlinear elliptic equations involving critical Sobolev exponents.}  Comm. Pure Appl. Math.  36 (1983), 437--477.
		
		\bibitem{Brezis Lieb lemma} H. Br\'{e}zis,  E. H. Lieb, {\it A relation between pointwise convergence of functions and convergence of functionals}. Proc. Amer. Math. Soc. 88 (1983), 486--490.
		
		\bibitem{Zou 2012}Z. J. Chen, W. M. Zou, {\it Positive least energy solutions and phase separation for coupled Schr\"{o}dinger equations with critical exponent.} Arch. Ration. Mech. Anal. 205 (2012), 515--551.
		
		\bibitem{Zou 2015} Z. J. Chen, W. M. Zou, {\it Positive least energy solutions and phase separation for coupled Schr\"{o}dinger equations with critical exponent: higher dimensional case.} Calc. Var. Partial Differential Equations. 52 (2015), 423--467.
		
		\bibitem{Carles=2018} R. Carles, I. Gallagher, {\it Universal dynamics for the defocusing logarithmic Schrodinger equation.}  Duke
		Math. J. 167 (2018), 1761--1801.
		
		\bibitem{Carles=2014} R. Carles, D. Pelinovsky, {\it On the orbital stability of Gaussian solitary waves in the log-KdV equation.}
		Nonlinearity. 27 (2014), 3185--3202.
		
		\bibitem{Colin=2004} M. Colin, L. Jeanjean, {\it Solutions for a quasilinear Schrodinger equation: a dual approach.}  Nonlinear
		Anal. 56 (2004), 213--226.
		
		\bibitem{Deng-He-Pan=Arxiv=2022} Y. B. Deng, Q. H. He, Y. Q. Pan, X. X. Zhong, {\it The existence of positive solution for an elliptic problem with critical growth and logarithmic perturbation.	} arXiv:2210.01373
		
		\bibitem{Esry1997} B. Esry, C. Greene, J. Burke,  J. Bohn,  {\it Hartree-Fock theory for double condesates.} Phys. Rev. Lett. 78 (1997), 3594--3597.
		
		\bibitem{Frantzeskakis2010}	D.J.  Frantzeskakis,  {\it Dark solitons in atomic Bose-Einstein condesates: from theory to experiments.} J. Phys. A.  43, 213001 (2010)
		
		\bibitem{Kivshar1998}  Y. S.  Kivshar, B.  Luther-Davies,   {\it Dark optical solitons: physics and applications.} Phys. Rep. 298 (1998), 81--197.
		
		\bibitem{Lin-Wei=CMP=2005} T. C. Lin, J. C. Wei, {\it Ground State of $N$ Coupled Nonlinear Schr\"{o}dinger Equations in $\R^n$,$n\leq$3}. Commun. Math. Phys. 255 (2005), 629--653.
		
		
		\bibitem{Lieb=2001} E. Lieb, M. Loss, {\it Analysis.}  Graduate Studies in Mathematics, 14. American Mathematical Society, Providence, RI, 2001.
		
		\bibitem{Liu-You-Zou=arxiv=2022}   T. H. Liu, S. You, W. M. Zou, {\it Least energy positive soultions for $d$-coupled Schr\"{o}dinger systems with critical exponent in dimension three.} arXiv:2204.00748.
		
		\bibitem{Poppenberg=2002}  M. Poppenberg, K. Schmitt, Z.-Q. Wang, {\it On the existence of soliton solutions to quasilinear Schrodinger
			equations.} Calc. Var. Partial Differential Equations. 14 (2002), 329--344.
		
		
		
		\bibitem{Sirakov 2007} B. Sirakov, {\it Least energy solitary waves for a system of nonlinear Schr\"{o}dinger equations in $\mathbb{R}^{n}$}. Comm. Math. Phys. 271 (1) (2007), 199--221.
		
		\bibitem{Talenti=AMPA=1976} G. Talenti, { \it Best constant in Sobolev inequality.} Ann. Mat. Pura Appl. 110 (1976), 353--372.
		
		\bibitem{Timmermans 1998}E. Timmermans, {\it Phase separation of Bose-Einstein condensates}. Phys. Rev. Lett. 81 (26) (1998), 5718--5721.
		
		\bibitem{Wang=ARMA=2019} Z.-Q. Wang, C. Zhang,{\it  Convergence from power-law to logarithm-law in nonlinear scalar field equations.}
		Arch. Rational Mech. Anal. 231 (2019), 45--61.
		
		
		\bibitem{W.Shuai=Nonlineariyu=2019} S. Wei, {\it Multiple solutions for logarithmic Schr\"{o}dinger equations.} Nonlinearity. 32 (2019), 2201--2225.
		
		\bibitem{william=1996} M. Willem. {\it Minimax Theorems.} Birkh\"{a}user Boston (1996).
		
		\bibitem{vazquez=AMO=1984} J. Vazquez, {\it A strong maximum principle for some quasilinear elliptic equations. } Appl. Math.	Optim., 12 (1984), 191--202.
		
		
		\bibitem{YePeng2014} H. Y. Ye, Y. F. Peng, {\it Positive least energy solutions for a coupled Schr\"{o}dinger system with critical exponent}. J. Math. Anal. Appl. 417 (2014), 308--326.
		
		
	\end{thebibliography}
\end{document}